\documentclass{amsart}
\usepackage[utf8]{inputenc}

\usepackage{booktabs}
\usepackage{todonotes}
\usepackage[authoryear, comma]{natbib}
\usepackage{geometry}
\usepackage{longtable} 
\usepackage{amsmath}
\usepackage{amsthm}
\usepackage{amsfonts}
\usepackage{amscd}
\usepackage{array}
\usepackage{bigstrut}
\usepackage{breqn}
\usepackage{tikz}
\usepackage{color}
\usepackage{listings}
\usetikzlibrary{decorations.pathreplacing}
\usepackage[center]{caption}
\usepackage{enumitem}
\usepackage{pgfplots}

\usepackage{subfigure}
\usepackage{tikz,wrapfig}
\usepackage{graphicx}
\usepackage{caption}
\usepackage[heightadjust=all,floatrowsep=none,captionskip=3pt]{floatrow}

\newtheorem{thm}{Theorem}[section]
\newtheorem{lem}[thm]{Lemma}

\theoremstyle{definition}
\newtheorem{defn}[thm]{Definition}
\newtheorem{ex}[thm]{Example}

\theoremstyle{remark}
\newtheorem*{rem}{Remark}

\newenvironment{alphafootnotes}
  {\par\edef\savedfootnotenumber{\number\value{footnote}}
   
   \setcounter{footnote}{0}}
  {\par\setcounter{footnote}{\savedfootnotenumber}}

\newcommand{\PP}{\mathbb{P}}
\newcommand{\C}{\mathbb{C}}

\newcommand{\Z}{\mathbb{Z}}
\newcommand{\Q}{\mathbb{Q}}

\newcommand{\Cstar}{\C^{\ast}}
\newcommand{\Zfin}[1]{\Z/#1\Z}  

\newcommand{\To}{\longrightarrow}

\newcommand{\abs}[1]{\left\vert#1\right\vert}  
\newcommand{\set}[1]{\left\{#1\right\}}  
\newcommand{\inn}[1]{\left\langle#1\right\rangle}  
\newcommand{\klat}{L_{\text{K3}}}  
\newcommand{\mirror}[1]{{#1}^\vee}

\newcommand{\pp}[4]{[#1:#2:#3:#4]}  
\newcommand{\p}{^{\prime}}

\newcommand{\Gmax}[1]{G_{#1}} 
\newcommand{\J}[1]{j_{#1}} 
\newcommand{\SLn}[1]{\SL({#1}, \C)}  
\newcommand{\SLgp}[1]{{\SL}_{#1}}  
\newcommand{\triv}{\{0\}}  
\newcommand{\A}[1]{{\mathbf A}_{#1}}  

\DeclareMathOperator{\Hom}{Hom}  
\DeclareMathOperator{\Aut}{Aut}  
\DeclareMathOperator{\SL}{SL}   
\DeclareMathOperator{\rk}{rank}  

\makeatletter
\def\imod#1{\allowbreak\mkern10mu({\operator@font mod}\,\,#1)}
\makeatother

\title{The mirror symmetry of K3 surfaces with non-symplectic automorphisms of prime order}
\author{Paola Comparin, Christopher Lyons, Nathan Priddis and Rachel Suggs}
\address{Paola Comparin, Laboratoire de Math\'ematiques et Applications, Universit\'e de Poitiers,
 T\'el\'eport 2, 
Boulevard Marie et Pierre Curie,
86962 Futuroscope Chasseneuil, France}
\email{paola.comparin@math.univ-poitiers.fr}
\address{Christopher Lyons, Department of Mathematics, University of Michigan, 530 Church Street, Ann Arbor, MI 48109, USA}
\email{lyonsc@umich.edu}

\address{Nathan Priddis, Department of Mathematics, University of Michigan, 530 Church Street, Ann Arbor, MI 48109, USA}
\email{priddisn@umich.edu}

\address{Rachel Suggs, Department of Mathematics, Brigham Young University, Provo, UT 84602, USA}
\email{rachel.suggs@byu.edu}

\begin{document}

\begin{abstract}
We consider K3 surfaces that possess a non--symplectic automorphism of prime order $p>2$ and we present, for these surfaces, a correspondence between the mirror symmetry of Berglund-H\"ubsch-Chiodo-Ruan and that for lattice polarized K3 surfaces presented by Dolgachev. 
\end{abstract}

\maketitle
\markboth{Paola Comparin, Christopher Lyons, Nathan Priddis and Rachel Suggs}{The mirror symmetry of K3 surfaces with non-symplectic automorphisms of prime order}

\section{Introduction}

Nearly twenty years ago, Berglund and H\"ubsch described a transposition rule for identifying mirror pairs of Calabi-Yau manifolds that are defined as hypersurfaces in weighted projective space \citep{berghub}. Subsequently, \citet{berghenn} and, independently,  \citet{krawitz} extended the rule to include a group of diagonal symmetries. More recently, \citet{BHCR} proved that this rule produces pairs of Calabi-Yau manifolds that are mirror to each other in the classical sense. When one has a Calabi-Yau manifold of dimension two, i.e., a K3 surface, a second type of mirror symmetry was summarized by Dolgachev in \citep{dolgachev}. In this paper we show that these two forms of mirror symmetry agree for a large class of K3 surfaces.
    
The extension of the Berglund-H\"ubsch transposition rule given by Berglund-Henningson and Krawitz is applicable to any quasihomogeneous invertible polynomial $W$ together with a group $G$ of diagonal automorphisms. Here, invertibility signifies that $W$ has the same number of monomials as variables. If the pair $(W,G)$ satisfies the Calabi-Yau condition and we let $Y_W=\set{W=0}$ in weighted projective space, then the orbifold $\left[Y_{W}/G \right]$ defines an (orbifold) K3 surface. The Berglund-H\"ubsch-Chiodo-Ruan (BHCR) mirror symmetry essentially says that for a well-defined polynomial $W^T$ and group $G^T$, the surfaces $[Y_{W}/G]$ and $[Y_{W^T}/G^T]$ are a mirror pair. 

The form of mirror symmetry discussed in \citep{dolgachev} applies to lattice polarized K3 surfaces, and is hereafter referred to as LPK3 mirror symmetry.  Beginning with a K3 surface $X$ polarized by a lattice $M$ (that satisfies some mild conditions), the construction produces a mirror family of dimension $20-\text{rank}(M)$ whose members are also lattice polarized K3 surfaces.  One may say that two lattice polarized K3 surfaces form an LPK3 mirror pair if each belongs to the mirror family of the other.  

If one starts only with an unpolarized K3 surface $X$, the standard choice is usually to polarize $X$ by the full Picard lattice. This is the approach taken by \citet{belcastro}. However, as can be seen from her work, polarizing the two members of a BHCR mirror pair by their Picard lattices does not, in general, produce an LPK3 mirror pair.  One hope for fixing this is to use smaller polarizing lattices, perhaps suggested by the particular forms of the surfaces.  

With this in mind, recall that a symplectic automorphism of a K3 surface is one that acts trivially on the canonical bundle. Several papers, including \citep{order_three}, \citep{other_primes}, \citep{nikulin_order2}, and \citep{Taki_order3}, classify all possible non-symplectic automorphisms of prime order $p$ based on their invariant lattice, a primitive sublattice of the Picard lattice. The connection between BHCR mirror symmetry and LPK3 mirror symmetry is made using the invariant lattice of a certain non-symplectic automorphism of prime order possessed by the K3 surfaces in question.
    
In \citep{involutions}, Artebani, Boissi\`ere, and Sarti use this idea for $p=2$ to prove that for a certain collection of K3 surfaces, a K3 surface and its BHCR mirror pair belong to LPK3 mirror families. The surfaces for which they proved this statement are those defined by equations of the form
\begin{equation*}
x_1^2 = f(x_2,x_3,x_4).
\end{equation*}
These surfaces possess the obvious involution $x_1 \mapsto -x_1$. This paper extends the Artebani-Boissi\`ere-Sarti result to surfaces of the form $x_1^p = f(x_2,x_3,x_4)$ for any prime $p$. Our main result is the following (see Sections \ref{backgrnd} and \ref{mirror} for notation): 

\begin{thm}\label{mainthm}
Consider the hypersurface $Y_{W}$ in a weighted projective space $\PP(w_1,w_2,w_3,w_4)$ given by a non--degenerate invertible polynomial of the form 
\begin{equation}\label{form_of_w}
x_1^p+f(x_2,x_3,x_4),
\end{equation}
where $p$ is prime, and let $G$ be any group of diagonal automorphisms of $W$ such that $J_W\subset G\subset SL_W$, where $J_W$ is the group generated by the exponential grading operator. Setting $\widetilde {G}:=G/J_W$ and $\widetilde{G^T}:=G^T/J_{W^T}$, the resolutions of the BHCR mirror orbifolds $[Y_W/\widetilde{G}]$ and $[Y_{W^T}/\widetilde{G^T}]$ belong to LPK3 mirror families.
\end{thm}

This result is the starting point for a larger program involving the Gromov-Witten theory of the quotient of a K3 surface by a non-symplectic automorphism, the monodromy group of the period domain of K3 surfaces and the construction automorphic forms.

By \citep{nikulin_finite}, any non-symplectic automorphism of a K3 surface of prime order $p$ will satisfy $p\leq 19$. A detailed analysis shows that no equation of the form of (\ref{form_of_w}) with $p=11,17$, or $19$ defines a minimal K3 surface. Since \citep{involutions} establishes Theorem \ref{mainthm} in the case $p=2$, in this paper we will only consider $p=3$, 5, 7, and 13.

Our strategy in proving this theorem is to first list all possible polynomials of the form (\ref{form_of_w}) that define K3 surfaces. We do this using a classification theorem due to Kreuzer and Skarke \citep{KrSk}, \citep{skarke} of invertible polynomials based on so-called atomic types. After finding all possible groups $G$ of automorphisms of $W$ that satisfy the hypotheses of BHCR symmetry, we then prove the theorem using the aforementioned results on non-symplectic automorphisms of prime order of K3 surfaces.

Finally, we mention that some related work has been done by Ebeling in \citep{ebeling}. In particular, in section 3, Ebeling shows that many of the LPK3 mirror pairs identified by \citet{belcastro} are also BHCR dual, provided they are of the form (\ref{form_of_w}) (with $p$ not necessarily prime) and that $\SL_W=J_W$. Some of the polynomials in Ebeling's work also show up here (precisely when $p$ is prime, and $S(\sigma_p)$ is equal to the Picard lattice). 

The structure of this paper is as follows. In Section \ref{backgrnd} we give some background and definitions necessary for understanding BHCR mirror symmetry, non-symplectic automorphisms of K3 surfaces, and lattice theory. In Section \ref{mirror} we explain in more detail two kinds of mirror symmetry---BHCR and LPK3---and we restate our main result concerning their compatibility. In section \ref{examples} we develop our method of proof by working all the computations needed to prove the result for several K3 surfaces. Finally in Section \ref{tables} we give several tables of calculations, which verify the theorem. In the Appendix we have reproduced several tables from \citep{other_primes} which are relevant for finding the invariant lattice of non-symplectic automorphisms of prime order. 

The results presented here were achieved simultaneously and independently by the first author and by collaboration of the latter three authors.  This joint paper emerged after each learned of the other's work.

\subsection{Acknowledgements}

We would like to thank Igor Dolgachev, Yongbin Ruan and Alessandra Sarti for many useful discussions and helpful insights. We would also like to thank Tyler Jarvis' students at BYU for providing a useful list of invertible quasihomogeneous polynomials, and for the use of their computer code for assisting in calculations. Finally, we thank Wolfgang Ebeling for pointing out to us the work \citep{berghenn}.  C.L.\ was partially supported by NSF RTG grant DMS 0943832. N.P.\ was supported by NSF RTG grant DMS 1045119. R.S.\ became involved as part of the REU program at the University of Michigan, supported by NSF RTG grant 053279.

\section{Background Material}\label{backgrnd}

In this section we present some background necessary to understand BHCR and LPK3 mirror symmetry.

\subsection{Quasihomogeneous polynomials and diagonal symmetries}\label{quasi_sec}
A map $W:\C^n\to \C$ is \emph{quasihomogeneous} of degree $d$ with integer weights $w_1, w_2, \dots, w_n$ if for every $\lambda \in \C$,
\[
W(\lambda^{w_1}x_1, \lambda^{w_2}x_2, \dots, \lambda^{w_n}x_n) = \lambda^dW(x_1,x_2, \dots, x_n).
\]
By rescaling the numbers $w_1, \dots, w_n$ and $d$, we can require that $\gcd(w_1, w_2, \dots, w_n)=1$. In this case, we say $W$ has the \emph{weight system} $(w_1, w_2, \ldots, w_n; d)$. 

A quasihomogeneous polynomial $W:\C^n  \rightarrow \C$ (sometimes called a potential in the physics literature) with a critical point at the origin is \emph{non-degenerate} if (i) the origin is the only critical point of $W$, and (ii) the fractional weights $\frac{w_1}{d}, \ldots, \frac{w_n}{d}$ of $W$ are uniquely determined by $W$. 
We say a non-degenerate quasihomogeneous polynomial is \emph{invertible} if it has the same number of monomials as variables.

If $W$ is invertible we can rescale variables so that $W = \sum_{i=1}^n \prod_{j=1}^n x_j^{a_{ij}}$. This polynomial can be conveniently represented by the square matrix $\A{W} = (a_{ij})$, which we will call the \emph{exponent matrix} of the polynomial. The conditions we have imposed on $W$ ensure that $\A{W}$ is an invertible matrix. In fact, the fractional weight $\frac{w_i}{d}$ is equal to the sum of the $i$-th row of $\A{W}^{-1}$.

\begin{thm}{\rm(cf. \citet[Theorem 1]{KrSk})} \label{atomic}
A quasi--homogeneous polynomial $W$ is non--degenerate and invertible if and only if it can be written as a direct sum of the three \emph{atomic types}:
\begin{itemize}
\item $W_{fermat} = x^a$
\item $W_{loop} = x_1^{a_1}x_2 + x_2^{a_2}x_3 + \ldots + x_n^{a_n}x_1$
\item $W_{chain} = x_1^{a_1}x_2 + x_2^{a_2}x_3 + \ldots + x_{n-1}^{a_{n-1}}x_n + x_n^{a_n}.$
\end{itemize}
Here the exponents are all integers greater than one.
\end{thm}


We will be interested in certain symmetries of the invertible polynomial $W$, which we call the \emph{group of diagonal symmetries} of $W$, denoted $G_W$. Precisely,
\begin{equation*}
\Gmax{W} = \{(c_1, c_2, \ldots, c_n) \in (\C^*)^n:  W(c_1x_1, c_2x_2, \ldots, c_nx_n) = W(x_1, x_2, \ldots, x_n)\}.
\end{equation*}

For $\gamma=(c_1, c_2, \ldots, c_n)\in \Gmax{W}$, the coordinates $c_1, c_2, \ldots, c_n$ must be roots of unity. This enables us to write the group $\Gmax{W}$ additively as a subgroup of $(\Q / \Z)^n$, with the identification 
\[
(c_1, c_2, \ldots, c_n) = (e^{2 \pi i g_1}, e^{2 \pi i g_2}, \ldots, e^{2 \pi i g_n}) \leftrightarrow (g_1, g_2, \ldots, g_n).
\]
In fact, if $\gamma=(g_1, g_2, \ldots, g_n)$ is an element of $\Gmax{W}$ written additively, the coordinates $g_j$ must satisfy $\sum_j a_{ij}g_j \in \Z$ for $1\leq i\leq n$. This is equivalent to the condition that  $\gamma$ is in the $\Z$-span of the columns of $\A{W}^{-1}$. Thus $\Gmax{W}$ is generated by the columns of $\A{W}^{-1}$. From this fact, one can check that the order of $\Gmax{W}$ is 
\[
|\Gmax{W}| = \det(\A{W}).
\]

As $W$ is quasihomogeneous, $\Gmax{W}$ always contains the \emph{exponential grading operator} $\J{W} = (e^{2 \pi i \frac{w_1}{d}}, e^{2 \pi i \frac{w_2}{d}}, \ldots, e^{2 \pi i \frac{w_n}{d}})$ (in additive notation $\J{W} = \left(\frac{w_1}{d}, \frac{w_2}{d}, \ldots, \frac{w_n}{d}\right)$). The group generated by $\J W$ is $J_W=\inn{j_W}$ and it is a cyclic group of order $d$. 

Finally, if we think of $\Gmax{W}$ multiplicatively, we can identify an element $\gamma =(c_1, c_2, \ldots, c_n)\in \Gmax{W}$ with a diagonal matrix that has the $c_i$ on the diagonal. If this matrix is contained in the group $\SLn{n}$, we abuse notation and write $\gamma\in\SLn{n}$. In additive notation (with $c_j = e^{2 \pi i g_j}$) this is equivalent to $\sum_j g_j \in \Z$. We define the group $\SLgp{W}$ to be $\Gmax{W}\cap \SLn{n}$.

We will be interested in pairs $(W,G)$, consisting of an invertible polynomial $W$ and a group of diagonal symmetries $G\subset G_W$, that will give rise to K3 surfaces. Toward this end, we say that the pair $(W, G)$ satisfies the \emph{Calabi-Yau} condition if both of the following hold:
\begin{itemize}
\item[(i)] $\displaystyle \sum_{i=1}^n w_i=d$,
\item[(ii)] $J_W\subset G\subset SL_W$.
\end{itemize}
Note that (i) ensures that $J_W\subset \SLgp{W}$. 

The following section will explain how the objects described in this section give rise to K3 surfaces.

\subsection{K3 Surfaces}\label{K3_sec}

Recall that a \emph{K3 surface} is a compact complex surface $X$ with trivial canonical bundle and $\dim H^1(X,\mathcal O_X)=0$. All K3 surfaces considered here will be projective and minimal. An automorphism of $X$ is \emph{symplectic} if the induced action on its canonical bundle is trivial, and is \emph{non--symplectic} otherwise.

In this paper, we will be concerned with a class of K3 surfaces that arises from certain hypersurfaces in weighted projective space. We start with the weighted projective space $\PP(w_1,w_2,w_3,w_4)$, assuming without loss of generality that it is \emph{normalized}, in the sense that $\gcd(w_1,w_2,w_3,w_4)=1$. A non-degenerate quasihomogeneous polynomial $W\in \C[x_1,x_2,x_3,x_4]$ with weight system $(w_1,w_2,w_3,w_4;d)$ defines a hypersurface $Y_W=\set{W=0}$ of degree $d$ in $\PP(w_1, w_2, w_3, w_4)$. The nondegeneracy of $W$ implies that the hypersurface $Y_W$ is \emph{quasismooth}.

In independent work, Reid (unpublished) and \citet{yonemura} have compiled a list of the 95 normalized weight systems $(w_1,w_2,w_3,w_4;d)$ (``the 95 families'') such that $\PP(w_1,w_2,w_3,w_4)$ admits quasismooth hypersurfaces of degree $d$ whose minimal resolutions are K3 surfaces. In all cases one has $d=\sum_i w_i$. By \citep[Theorem 1.13]{hypergeometric}, the general hypersurface of degree $d$ in one of these 95 weight systems is also \emph{well-formed}, meaning that it does not contain any of the coordinate lines $\set{x_i=x_j=0}$. Moreover, any quasismooth well-formed hypersurface in the normalized weighted projective space $\PP(w_1,w_2,w_3,w_4)$ of degree $d=\sum_i w_i$ has a K3 surface as its minimal resolution by \citep[Lemma 1.12]{hypergeometric}. Thus, if $W$ is an invertible polynomial with weight system $(w_1,w_2,w_3,w_4;d)$ belonging to one of the 95 families, then the minimal resolution $X_W$ of $Y_W$ is a K3 surface.

The group $\Gmax{W}$ acts in an obvious manner by automorphisms on the surface $Y_W$, and by extension on the K3 surface $X_W$. By \citep[Proposition 1]{involutions}, an automorphism $\sigma\in \Gmax{W}$ is symplectic if and only if $\det \sigma =1$, that is, if and only if $\sigma\in \SLgp{W}$.

We now add the condition that $W$ be of the form (\ref{form_of_w}) for a prime $p$. In this case, $X_W$ has an obvious non-symplectic automorphism of order $p$, arising from the map
\[
\sigma_p:\pp{x_1}{x_2}{x_3}{x_4}\mapsto \pp{\zeta_px_1}{x_2}{x_3}{x_4}. 
\]
We can identify this map as multiplication by an element of $\Gmax{W}$, which we will also denote as $\sigma_p$. In additive notation this is $\sigma_p=\left(\tfrac 1p,0,0,0\right)$.

We may further obtain K3 surfaces from $X_W$ by taking quotients.  Indeed, if $J_{W} \subset G\subset \SLgp{W}$, then $\widetilde{G}=G/J_{W}$ acts symplectically on the K3 surface $X_W$, and we may consider the quotient $X_W/\widetilde{G}$.  The minimal resolution of $X_W/\widetilde{G}$, which we will denote as $X_{W,G}$, is also a K3 surface.  The non--symplectic automorphism $\sigma_p$ descends to $X_W/\widetilde{G}$, and thus one obtains a non-symplectic automorphism of order $p$ on $X_{W,G}$; we will also denote the automorphism on $X_{W,G}$ as $\sigma_p$ when there is no risk of confusion.  We remark that, as $J_W$ acts trivially on $X_W$, we have $X_W=X_{W,J_W}$.

In the sequel, we will be concerned exclusively with K3 surfaces arising from the construction above, and for brevity we will summarize it by introducing the following terminology. Let $W$ be an invertible polynomial with weight system $(w_1,w_2,w_3,w_4;d)$ and let $ G \subset \Gmax{W}$. Suppose that $(W, G)$ satisfies the Calabi-Yau condition (see (i) and (ii) in Section \ref{quasi_sec}), and that the following additional conditions hold:
\begin{itemize}
\item[(iii)] $W$ is of the form (\ref{form_of_w}), and
\item[(iv)] the weight system $(w_1,w_2,w_3,w_4;d)$ belongs to one of the 95 families,
\end{itemize}
Then as in the discussion above we have a K3 surface $X_{W,G}$ that is the minimal resolution of $X_W/\widetilde{G}$. We will call $X_{W, G}$ a \emph{$p$-cyclic K3 surface} because it is equipped with the non--symplectic automorphism $\sigma_p$ coming from the diagonal symmetry $\left(\frac 1p, 0,0,0\right) \in \Gmax{W}$.

Next we describe restrictions on the prime exponent $p$ in (\ref{form_of_w}). As noted in the introduction, for general reasons one must have $p\leq 19$. But in fact the specific nature of (\ref{form_of_w}) further limits the possibilities for $p$. Indeed, among the 95 families, one may show by direct methods (see Example \ref{weight_ex}) that only 69 of these weight systems admit quasihomogenous invertible polynomials of the form (\ref{form_of_w}), and only with the primes $p=2,3,5,7,13$.  If we further limit ourselves to $p\neq 2$ (as will be our focus), then the number of weight systems is reduced to 41.  All possibilities for $W$ that may occur for $p\neq 2$ are enumerated in terms of the exponent $p$ in Tables \ref{tab_eq3} -- \ref{tab_eq13}. As can be seen, many of these weight systems admit multiple polynomials and some of the polynomials appear in more than one table.

\begin{ex}\label{weight_ex}  Let us show that the weight system $(w_1,w_2,w_3,w_4;d)=(5,4,3,3;15)$ appears among the 41 admissible families. 
We look for an invertible polynomial $W$ of the form \eqref{form_of_w} whose weight system is $(5,4,3,3;15)$.
If $W$ is as in \eqref{form_of_w}, then $pw_1=d$ holds. 
If $w_1$ does not satisfy $pw_1=d$ we can perform a change of coordinates $x_1\leftrightarrow x_j$ such that $pw_j=d$.
In this case the form of $W$ will contain the term $x_j^p$. For this weight system, there are two possible choices for $p$, i.e. $p=3,5$.

We start by considering $p=3$. Thus $W=x_1^3+f(x_2,x_3,x_4)$ and admits the automorphism $\sigma_3=(\frac13,0,0,0)$. By Theorem \ref{atomic}, the polynomial $W$ is a sum of atomic types and thus $f$ is one of the following:
\[f(x_2,x_3,x_4) =  \begin{cases}
 x_2^a+x_3^b+x_4^c& \text{(fermat)}  \\
 x_2^ax_3+x_3^bx_4+x_4^c & \text{(chain) }  \\
 x_2^ax_3+x_3^bx_4+x_4^cx_2& \text{(loop)}  \\
 x_2^ax_3+x_3^b+x_4^c& \text{(chain+fermat)} \\
x_2^ax_3+x_3^bx_2+x_4^c& \text{(loop+fermat)} \\
\end{cases} \]
for certain non-zero $a,b,c\in\mathbb N$.
Determining whether there is a polynomial $f$ of each of these types reduces to solving a linear system. There are two possibilities: $W_1=x_1^3+x_2^3x_3+x_3^4x_4+x_4^5$ and $W_2=x_1^3+x_2^3x_3+x_3^5+x_4^5$.

When we pass to $p=5$, we find that the weight $w_j$ associated to the term $x_j^5$ must be 3. 
Hence we choose $x_1\leftrightarrow x_4$ and obtain the polynomial $x_1^3+x_2^3x_3+x_3^5+x_4^5$, which is $W_2$ from above. Here the diagonal symmetry of order 5 is $(0,0,0,\frac15)$.
\end{ex}

In the sequel we will make a change of notation from $(x_1,x_2,x_3,x_4)$ to $(x,y,z,w)$ for the variables of $W$, with the weights in nonincreasing order from left to right.  This convention is also used Tables \ref{tab_eq3} -- \ref{tab_eq13}.

\subsection{Lattice theory}\label{lattice_sec}

By a \emph{lattice}, we shall mean a free abelian group $L$ of finite rank equipped with a non-degenerate symmetric bilinear form $B\colon L \times L \to \Z$. A lattice $L$ is \emph{even} if the associated quadratic form $B(x,x)$ takes values in $2\Z$. The \emph{signature} of $L$ is the signature $(t_+,t_-)$ of $B$. A lattice is \emph{hyperbolic} if its signature is $(1,\rk(L)-1)$. A sublattice $L'\subseteq L$ is called \emph{primitive} if $L/L'$ is free.

As the bilinear form $B$ is non-degenerate, it induces an embedding $L \hookrightarrow L^\ast$, where $L^\ast= \Hom(L,\Z)$.  The \emph{discriminant group} $A_{L}= L^\ast/L$ is finite. The minimal number of generators of $A_L$ is called the \emph{length} of $L$. Note that if one writes $B$ as a symmetric matrix in terms of a basis of $L$, then the order of $A_{L}$ is equal to $\abs{\det(B)}$. If $A_L = \triv$, $L$ is called \emph{unimodular}.  For a prime number $p$, $L$ is called \emph{$p$-elementary} if $A_L \simeq (\Z/p\Z)^a$ for some $a$; in this case, $a$ is the length of $A_L$. 

\citet{rudakov} have given a complete classification of even, indefinite, $p$-elementary lattices. It states in particular that for $p\neq 2$, an even indefinite $p$-elementary lattice of rank $r\geq 2$ is uniquely determined by the integer $a$. In the case $p=2$ a third invariant $\delta\in\{0,1\}$ is necessary to identify the lattice, though this will not be relevant in this paper.

We follow the notation of \citep[Section 1]{other_primes} to present some particular lattices that are relevant to our calculations.  The lattice $U$ is the unimodular hyperbolic lattice of rank 2. For a properly chosen basis the bilinear form is given by the matrix
\[
\left(\begin{matrix}
0 & 1\\
1 & 0
\end{matrix}\right).
\]

The lattice $A_{n}$ is the negative definite even lattice whose discriminant group has order $n+1$.  It is associated to the Dynkin diagram $A_n, n\geq 1$. In particular, if $p$ is prime, then $A_{p-1}$ is $p$-elementary with $a=1$.  Likewise, $E_8$ is the negative definite unimodular even lattice of dimension 8, and is associated to the Dynkin diagram $E_8$.

If $p\equiv3\imod 4$, we denote by $K_p$ the lattice corresponding to the matrix
\[
\left(\begin{matrix}
-(p+1)/2&1\\
1&-2
\end{matrix}\right).
\]
It is a negative-definite, $p$-elementary lattice with $a=1$.

If $p\equiv 1\imod 4$  the lattice $H_p$ corresponding to the matrix
\[
\left(\begin{matrix}
(p-1)/2&1\\
1&-2
\end{matrix}\right)
\]
is a hyperbolic, $p$-elementary lattice of length 1.

Given a lattice $L$, we denote by $L(n)$ the lattice obtained by multiplying the bilinear form by $n$.

Let $X$ be a K3 surface. We have the cup product form
\[
H^2(X,\Z) \times H^2(X,\Z) \stackrel{\cup}{\To} H^4(X,\Z) \  \tilde{\To} \ \Z,
\]
where the isomorphism on the right sends the class of a point to $1\in\Z$. This pairing makes $H^2(X,\Z)$ into an even unimodular lattice of signature $(3,19)$.  As such, it is isometric to the \emph{K3-lattice}
\[
\klat = U^3\oplus (E_8)^2.
\]
Here and in the sequel, direct sum of lattices is used to denote orthogonality. 

We let
\[
S_X = H^2(X,\Z) \cap H^{1,1}(X,\C)
\]
denote the Picard lattice of $X$ in $H^2(X,\Z)$ and $T_X = S_X^\perp$ denote the transcendental lattice.

Any non-symplectic automorphism $\sigma\in\Aut(X)$ induces an isometry $\sigma^\ast\in\Aut(H^2(X,\Z))$.  We let $S(\sigma)\subseteq H^2(X,\Z)$ denote the $\sigma^\ast$-invariant sublattice of $H^2(X,\Z)$, which one can check is a primitive sublattice of $H^2(X,\Z)$. We let $T(\sigma) = S(\sigma)^\perp$ denote its orthogonal complement, which is also primitive.  Note that, as $\sigma$ is non-symplectic, the Lefschetz $(1,1)$ Theorem implies that
\[
S(\sigma) \subset S_X.
\]
Finally, as the sum over the $\inn{\sigma}$-orbit of an ample divisor class is a $\sigma$-invariant ample class by Nakai-Moishezon, the Hodge Index Theorem implies that the signature of $S(\sigma)$ is $(1,t)$ for some $t\leq 19$, i.e. $S(\sigma)$ is hyperbolic.

\begin{lem}{\rm{(\citet{nikulin_finite}, \citet[Theorem 2.1]{other_primes})}} Let $X$ be a K3 surface with a non--symplectic automorphism $\sigma$ of prime order $p$. Then $S(\sigma)$ and $T(\sigma)$ are $p$-elementary lattices, and $A_{S(\sigma)}\cong A_{T(\sigma)}\cong (\Z/p\Z)^a$ with $a\leq \frac{\rk(T(\sigma))}{p-1}.$
\end{lem}

We will let $r$ denote the rank of $S(\sigma)$. Then we have a set of invariants $(r,a)$ of the lattice $S(\sigma)$. In addition to $r$ and $a$, some authors also work with $m=\frac{22-r}{p-1}$, which is the rank of $T(\sigma)$ as a $\Z[\zeta_p]$-module, for a primitive $p$-th root of unity $\zeta_p$. In this paper our main focus with be on $r$ and $a$, but we will make some use of $m$  (e.g., see Theorem \ref{summary_invariants} and the discussion at the end of Section \ref{sec_LPK3}). Note that the lemma says $a\leq m$. 

Let $X^\sigma$ denote the fixed locus of $\sigma$. In \citep{other_primes}, the authors prove that there is a close connection between the invariants of $S(\sigma)$ and the topological structure of $X^\sigma$. In fact, each determines the other, as demonstrated in the following theorem. We omit the case $p=2$, as we are primarily interested in the cases $p=3,5,7,13$.

\begin{thm}[cf. \citet{other_primes}]\label{summary_invariants}
Let $X$ be a K3 surface with a non-symplectic automorphism $\sigma$ of prime order $p\neq 2$. Then the fixed locus $X^\sigma$ is nonempty, and consists of either isolated points or a disjoint union of smooth curves and isolated points of the following form:
\begin{equation}
X^\sigma=C\cup R_1\cup\ldots\cup R_k\cup \{p_1,\ldots,p_n\}\label{eq_fixed_locus}.
\end{equation}
Here $C$ is a curve of genus $g\geq 0$, $R_i$ are rational curves and $p_i$ are isolated points. 

Furthermore, if $X^\sigma$ contains a curve and $S(\sigma)$ has invariants $(r,a)$, then the following hold:
\begin{itemize}
\item $m=2g+a$;
\item if $p=3$ then $1-g+k=(r-8)/2$ and $n=10-m$ ;
\item if $p=5$ then $1-g+k=(r-6)/4$, $n=16-3m$;
\item if $p=7$ then $1-g+k=(r-4)/6$ and $n=18-5m$;
\item if $p=13$ then $(g,n,k)=(0,9,0)$ and $S(\sigma)=H_{13}\oplus E_8$.
\end{itemize}
\end{thm}

\begin{rem} In \citep{order_three} this theorem is rendered slightly differently for the case $p=3$. The difference there is that the invariant $k$ represents the total number of fixed curves, including the curve with (possibly) positive genus. Here and in \citep{other_primes}, the invariant $k$ does not include the curve $C$. We also point out the the case $g=0$ does occur. In that case the curve $C$ is also rational, and the total number of fixed rational curves is $k+1$. (See Example \ref{ex2} for an example.)
\end{rem}

Tables \ref{tab_lattices3} -- \ref{tab_lattices13} in Appendix \ref{tables_lattices} contain a complete classification of the lattices $S(\sigma)$ and $T(\sigma)$ according to their invariants when $\sigma$ is of order $3,5,7$ or $13$.

Note that there is only a single possibility for the isometry class of $S(\sigma)$ when $p=13$.  Since this fact renders our main result (see Theorem \ref{main_thm}) trivial when $p=13$, we focus mainly on the cases $p=3,5,7$.

\section{Mirror Symmetry}\label{mirror}

\subsection{LPK3 mirror symmetry}\label{sec_LPK3}

In this section we describe relevant aspects of mirror symmetry for lattice polarized K3 surfaces. As mentioned, we will refer to this simply as LPK3 mirror symmetry.  For a more complete treatment, we refer the reader to \citep{dolgachev}.

Suppose that $M$ is an even lattice of signature $(1,t)$ with $t\leq 18$ that embeds primitively into $\klat$.  Let $\iota:M\hookrightarrow \klat$ be such an embedding and suppose further that there is a primitive embedding $\iota':U\hookrightarrow \iota(M)^\perp$ of the lattice $U$ into the orthogonal complement of $M$ in $\klat$. Under these assumptions, \citep[Corollary 1.13.3]{nikulin_ISBF} ensures that the isometry class of the orthogonal complement of $\iota'(U)$ inside $\iota(M)^\perp$ is in fact independent of the choices $\iota, \iota'$.  This observation ensures that the following is well-defined:

\begin{defn}\label{mirror_defn}
Let $M$ be a lattice of signature $(1,t)$ with $t\leq 18$ that admits a primitive embedding $\iota:M\hookrightarrow \klat$. We will say $M$ is \emph{mirror-hyperbolic} if its orthogonal complement $\iota(M)^\perp$ in $\klat$ admits a primitive embedding $\iota':U \hookrightarrow \iota(M)^\perp$.  For a mirror-hyperbolic lattice $M$, we define (up to isometry) the \emph{mirror lattice} $\mirror{M}$ of $M$ via the orthogonal decomposition
\[
\iota(M) = \iota'(U) \oplus \mirror{M}.
\]
\end{defn}

Note that $\mirror{M}$ also embeds primitively into $\klat$ and has signature $(1,19-t)$.  By \citep{nikulin_ISBF}, the discriminant groups of $M$ and $\mirror{M}$ are isomorphic.  Thus if $M$ is p-elementary, with invariants $(r,a)$, then $M^\vee$ is also $p$-elementary with invariants $(20-r,a)$.  One easily checks that $\mirror{M}$ is also mirror-hyperbolic and that the mirror lattice of $M^\vee$ is $M$. \footnote{This definition of the mirror lattice uses a more restrictive hypothesis than that in \citep{dolgachev}; in the terminology used there, we are assuming that $M^\perp$ contains an isotropic $1$-admissible vector.  Moreover, our definition in this restricted setting is coarser than the one used by Dolgachev, since we do not keep track of the embedding $U\hookrightarrow M^\perp$ and instead only consider $M^\vee$ up to isometry.}

Now let $X$ be a K3 surface and suppose that $M$ is a lattice of signature $(1,t)$.  If $j\colon M\hookrightarrow S_X$ is a primitive embedding into the Picard lattice of $X$, one calls the pair $(X,j)$ an \emph{$M$-polarized K3 surface}.  More coarsely, we will call the pair $(X,M)$ an \emph{$M$-polarizable} K3 surface if such an embedding $j$ exists.  Note for an $M$-polariziable K3 surface $(X,M)$, the lattice $M$ necessarily embeds primitively into $\klat$.

\begin{defn}
Let $(X,M)$ be an $M$-polarizable K3 surface and $(X',M')$ be an $M'$-polarizable K3 surface, where $M$ and $M'$ are both mirror-hyperbolic.  We will say that $(X,M)$ and $(X',M')$ are \emph{LPK3 mirrors} if $M' = \mirror{M}$ (or equivalently $M=\mirror{(M')}$).
\end{defn}

\begin{rem}
Let $X_{W,G}$ be a $p$-cyclic K3 surface, with its non-symplectic automorphism $\sigma_p$.  We have noted that the invariant lattice $S(\sigma_p)$ embeds primitively into the Picard lattice, yielding the polarizable K3 surface $(X_{W,G},S(\sigma_p))$.  We will also see in Lemma \ref{Tspect} that $S(\sigma_p)$ is mirror-hyperbolic.  If $p=3,5,7$, then our results show that the notion of the mirror lattice of $S(\sigma_p)$ has an elegant formulation in terms of a certain reflection, as we now describe.

By \citep[Theorem 0.1]{other_primes}, the invariant lattice $S(\sigma_p)$ is determined, up to isometry, by the pair of invariants $(r,a)$ associated to $S(\sigma_p)$.  Recalling that $m=\frac{22-r}{p-1}$, the same is then true of the pair $(m,a)$.  When one plots in the $(m,a)$-plane those values that are realized as the invariants of $S(\sigma)$, for some non-symplectic automorphism $\sigma$ of order $p$ on some K3 surface $X$, one arrives at the figures in Appendix \ref{tables_lattices}.  One sees that each figure is nearly symmetric about the line $m=\mu/2$, where
\[
\mu = \frac{24}{p-1} =
\begin{cases}
12 & \text{if } p=3 \\
6 & \text{if } p=5 \\
4 & \text{if } p=7. \\
\end{cases}
\]
The reflection through this line is given by the involution $(m,a)\mapsto (\mu-m,a)$.  Using instead the invariants $(r,a)$, this involution takes the form $(r,a)\mapsto (20-r,a)$.  Note that, as mentioned after Definition \ref{mirror_defn}, this is the same involution relating the invariants of a mirror-hyperbolic lattice $M$ to those of its mirror $M^\vee$.

Studying the figures in Appendix A, one sees that only a few anomalous values (by which we mean that $(m,a)$ appears but $(\mu-m,a)$ does not) keep this involution from being a true symmetry of the figure.  One facet of our results is that none of these anomalous values arise from $p$-cyclic K3 surfaces.  (This will follow from (i) Lemma \ref{Tspect} below and (ii) the fact that $\sigma_3$ always fixes at least one curve on a $3$-cyclic K3 surface, ruling out the invariants $(r,m,a)=(8,7,7)$ for $S(\sigma_3)$).  The conclusion is then that in the case of $p$-cyclic K3 surfaces, the mirror lattice of $S(\sigma_p)$, with invariants $(m,a)$, is exactly the lattice obtained by reflection across the line $m=\mu/2$.
\end{rem}

\subsection{BHCR mirror symmetry}\label{BHCR_sec}
The second formulation of mirror symmetry that we consider comes from the Landau-Ginzburg model for mirror symmetry. This particular formulation of mirror symmetry was developed initially by \citet{berghub}, and later refined by \citet{berghenn} and \citet{krawitz}. 

General Landau-Ginzburg mirror symmetry is defined for a large class of pairs $(W,G)$, where $W$ is an invertible polynomial and $G \subset \Gmax{W}$ is a group of diagonal symmetries. Howevever, we will restrict ourselves to the case where $(W, G)$ satisfies the Calabi-Yau condition. 

Given such a pair, we will define the mirror pair $(W^T,G^T)$ in the following way. First, if $W=\sum_{i=1}^n \prod_{j=1}^n x_j^{a_{ij}}$, we define $W^T = \sum_{i=1}^n \prod_{j=1}^n x_j^{a_{ji}}$. In other words, $W^T$ is the polynomial corresponding to the transpose of the exponent matrix $\A{W}$ of $W$. That $W^T$ is an invertible polynomial follows from the classification in Theorem \ref{atomic}.

Next, using additive notation we define the dual group $G^T$ of $G$ as 
\begin{equation}\label{dualG_def}
G^T = \set{\; g \in \Gmax{W^T} \; | \;\;g\A{W} h^T \in \Z \text{ for all } h \in G \; }.
\end{equation}

From \citet[Proposition 3]{involutions} we have the following useful properties of the dual group.
\begin{itemize}
\item $(G^T)^T = G$
\item If $G_1\subset G_2$, then $G_2^T\subset G_1^T$ and $G_1/G_2\cong G_2^T/G_1^T$. 
\item $(\Gmax{W})^T = \triv$
\item $(J_W)^T=\SLgp{W^T}$
\item In particular, if $J_W\subset G$, then $G^T\subset \SLgp{W}$.
\end{itemize}


It remains to verify that the pair $(W^T, G^T)$ also satisfies the Calabi-Yau condition.  Note that since $W$ satisfies $\sum w_i = d$, the transpose polynomial $W^T$ does as well (this is because $\tfrac{w_i}{d}$ is the sum of the $i$th row of the matrix $\A{W}^{-1}$). Also, the properties of $G^T$ listed above tell us that $J_{W^T} \subset G^T \subset \SLgp{W^T}$. We conclude that $(W^T, G^T)$ indeed satisfies the Calabi-Yau condition.


The mirror construction that we have just described was initially used to identify mirror pairs of Landau-Ginzburg models constructed from $(W, G)$ and $(W^T, G^T)$, respectively. Because the groups $\widetilde G = G/J_W$ and $\widetilde{G^T} = G^T/J_{W^T}$ act on the hypersurfaces $Y_W=\set{W=0}$ and $Y_{W^T}=\set{W^T=0}$, respectively, we can investigate the Calabi-Yau orbifolds $[Y_{W}/\widetilde{G}]$ and $[Y_{W^T}/\widetilde{G^T}]$. \citet{BHCR} established what is known as the Landau-Ginzburg/Calabi-Yau correspondence. Using this correspondence Chiodo and Ruan proved the following theorem, thus establishing a close connection between Landau-Ginzburg mirror symmetry and mirror symmetry for Calabi-Yau orbifolds.

\begin{thm}{\rm{(\citet[Theorem 2]{BHCR})}}\label{CR_thm}
The orbifolds $[Y_{W}/\widetilde{G}]$ and $[Y_{W^T}/\widetilde{G^T}]$ form a mirror pair; that is,
\[
H_{CR}^{p,q}([ Y_{W}/\widetilde{G}],\C) \cong H_{CR}^{n-2-p,q}([ Y_{W^T}/\widetilde{G^T}],\C)
\]
where $H_{CR}(-, \C)$ indicates Chen-Ruan orbifold cohomology.
\end{thm}

Now consider a $p$-cyclic K3 surface $X_{W,G}$, so that in particular the weight system of $W$ belongs to one of the 95 families.  Then we have seen that $(W^T, G^T)$ also satisfies the Calabi-Yau condition. A direct check shows that the weight system of $W^T$ belongs to the 95 families. Hence $X_{W^T,G^T}$ is also a $p$-cyclic K3 surface. In light of Theorem \ref{CR_thm}, following \citep{involutions}, we define the \emph{BHCR mirror} of $X_{W,G}$ to be $X_{W^T,G^T}$. 

We have now described two kinds of mirror symmetry, and we expect some form of compatibility in situations where both apply. The next section states our main theorem, which makes precise the sense in which BHCR and LPK3 mirror symmetry are compatible.

\subsection{Main Theorem}\label{main_sec}
Our goal is to show compatibility of BHCR and LPK3 mirror symmetry for $p$-cyclic K3 surfaces $X_{W,G}$ where $p\neq 2$. Clearly, in order to talk about such compatibility, we must first verify that an LPK3 mirror exists for each such surface. This is accomplished by the following lemma.

\begin{lem}\label{Tspect}
If $p=3,5,7,13$, then the lattice $S(\sigma_p)$ is mirror-hyperbolic. 
\end{lem}

\begin{proof}
This amounts to showing that $U$ embeds primtively into $T(\sigma_p)$.  Upon consulting Tables \ref{tab_lattices3} -- \ref{tab_lattices13}, this is equivalent to showing that the invariants $(r,a)$ of $S(\sigma_p)$ do not take on the following values:
\begin{itemize}
\item For $p=3$: $(r,a) = (20,1)$
\item For $p=5$: $(r,a) = (6,4)$
\item For $p=7$: $(r,a) = (4,3)$.
\end{itemize}
This in turn follows from inspection of Tables \ref{tab_eq3} -- \ref{tab_eq13}.
\end{proof}

We are now ready to state our main result, which is the following:

\begin{thm}\label{main_thm}
Consider a $p$-cyclic K3 surface $X_{W,G}$ with its non-symplectic automorphism $\sigma_p$ of prime order, along with its BHCR mirror $X_{W^T, G^T}$ with its non-symplectic automorphism, which we will denote as $\sigma_p^T$. Then $\left(X_{W,G}, S(\sigma_p)\right)$ and $\left(X_{W^T,G^T}, S(\sigma_p^T\right))$ are LPK3 mirrors. 
\end{thm}

\begin{proof}
The theorem will be verified if $\mirror{S(\sigma_p)}=T(\sigma_p)/U\cong S(\sigma_p^T)$. This in turn follows by checking that the invariants of $S(\sigma_p)$ and $S(\sigma_p^T)$ are $(r,a)$ and $(20-r,a)$, respectively. Thus, at the heart of this proof is the determination of the isometry class of $S(\sigma_p)$ for all $p$-cyclic K3 surfaces $X_{W,G}$ with $p\neq 2$. Detailed illustrations of the methods behind this determination are given in Section \ref{ex_sec}, but the brief idea is to find the topological invariants of the fixed locus of $\sigma_p$ acting on $X_{W,G}$, which by Theorem \ref{summary_invariants} will determine the isometry classes of $S(\sigma_p)$. A complete list of all possibilities for $W$ is found in Tables \ref{tab_eq3} -- \ref{tab_eq13} (and we have shown in Section \ref{K3_sec} how to obtain them).  These same tables list the invariants $(r,a)$ of $S(\sigma_p)$.  Thus, our theorem follows from inspection of Tables \ref{tab_eq3} -- \ref{tab_eq13}, in the following manner.

First, in all but one exceptional case (namely No.\ 3a in Table \ref{tab_eq3}), the group $\SL_W/J_W$ is cyclic, and hence a subgroup $J_W\subset G\subset \SL_W$ is uniquely determined by the order of $G/J_W$.  Thus given a $p$-cyclic K3 surface $X_{W,G}$, one finds the row of the table corresponding to $(W,G)$, and from this one may (in the non-exceptional cases) read off the invariants $(r,a)$ of $S(\sigma_p)$ and the number of $W^T$ (labelled as ``BHCR dual'').  Noting that
\[
|G^T/J_{W^T}| = |\SL_W/G| = \frac{|\SL_W/J_W|}{|G/J_W|},
\]
one may then find the row corresponding to $(W^T,G^T)$ and verify the aforementioned symmetry between the invariants of $S(\sigma_p)$ and $S(\sigma_p^T)$.

More details about the exceptional case No.\ 3a are found Example \ref{ex3}.
\end{proof}

\section{Examples}\label{examples}

\label{ex_sec}
To illustrate the techniques used to compute the invariants given in Tables \ref{tab_eq3} -- \ref{tab_eq13}, we will give some examples. All cases are similar to those highlighted here. We begin with a few results that we will use later. 

We start with an equation $x_1^p+f(x_2,x_3,x_4)$ in $\PP(w_1,w_2,w_3,w_4)$, with the non--symplectic automorphism $\sigma_p$.
In order to describe the surface $Y_W=\set{W=0}\subset\PP(w_1,w_2,w_3,w_4)$ more explicitly, we will find the singular points. Since $Y_W$ is quasismooth, the singular points $Y_W$ only occur at singular points of $\PP(w_1,w_2,w_3,w_4)$. To find them, we examine the action of $\Cstar$ on $\C^4$, which is given by 
\[
\lambda\cdot (x,y,z,w)=(\lambda^{w_1}x, \lambda^{w_2} y, \lambda^{w_3} z, \lambda^{w_4} w).
\] 
The singularities in question will be occur at points $\pp xyzw\in Y_W$ for which two or three of the coordinates vanish and for which the action of $\Cstar$ has non-trivial isotropy. Resolution of these points will yield $X_{W,J_W}$. Exceptional curves of the resolution will often contribute to the fixed locus of $\sigma_p$. 

Recall that we can calculate $(r,a)$ if we know the topological invariants $(g,n,k)$. We can find $n$ and $k$ explicitly from the action of $\sigma_p$. In order to compute $g$, we need the following lemma, which can be found, e.g., in \citep[Theorem 12.2]{fletcher}.

\begin{lem}\label{eq_genus} The genus of a quasismooth curve of degree $d$ in $\PP(w_1,w_2,w_3)$ is 
\begin{equation*}
g(C)=\frac{1}{2}\left(\frac{d^2}{w_1w_2w_3}-d\sum_{i>j\geq1}\frac{\gcd(w_i,w_j)}{w_iw_j}+\sum_{i=1}^3\frac{\gcd(w_i,d)}{w_i}-1\right).
\end{equation*}
\end{lem}

With these results in mind, we will now give several examples. The first example, in which $W^T=W$, will give the most detail. The second example differs only in that $W^T\neq W$. In the third example, we resolve the ambiguity in the K3 surface numbered 3a from Table \ref{tab_eq3} by computing each subgroup of $\SL_W$ and its transpose. 

\begin{ex}\label{ex1}
We will work through the calculations for the K3 surfaces arising from the polynomial $W=x^2 + y^3+z^8 + w^{24}$ in detail. This is No.\ 13d in Table \ref{tab_eq3}. Note that the order of the variables has been changed to match the table, so that $x_1$ in (\ref{form_of_w}) corresponds with the variable $y$ in $W$. 

The weight system for $W$ is $(12,8,3,1;24)$, and so in additive notation we have $j_{W} = \left(\frac{1}{2}, \frac{1}{3}, \frac{1}{8}, \frac{1}{24}\right)$. The exponent matrix is
\[
\A{W} = 
\left( \begin{array}{cccc}
2 & 0 & 0 & 0 \\
0 & 3 & 0 &0 \\
0 & 0 & 8 &0 \\
0 & 0 & 0 & 24\end{array} \right). 
\]
Since $\A{W}$ is symmetric, we see that $W^T=W$.

We need to find all subgroups $G\subset \Gmax{W}$ satisfying $J_W \subset G \subset \SL_{W}$. By Section \ref{BHCR_sec} we know $\Gmax{W}/\SLgp{W}\cong J_{W^T}\cong J_{W}$.
By Section \ref{quasi_sec} we have $|\Gmax{W}|=\det (A_{W})=1152$. As $|J_W|=24$, we can see that $|\SLgp{W}/J_W| = 2$ as in Figure \ref{group_lattice}. So there are two possible choices for the group $G$, namely $J_W$ and $\SLgp{W}$. 
\begin{figure}[h]
\centering
\begin{tikzpicture}
[xscale = 1, yscale = 1]
    \draw (1,.2) --(1,.8) ;
    \draw (1,1.2) --(1,1.8) ;
    \draw (1,2.2) --(1,2.8) ; 
    \node at (1,0) {$\triv$};
    \node at (1,1) {$J_{W}$};
    \node at (1,2) {$\SL_{W}$};
    \node at (1,3) {$G_W$};
    \draw [decorate,decoration={brace,amplitude=9pt},xshift=-12pt,yshift=0pt] (1,0) -- (1,3)node [black,midway,xshift=-25pt] { $1152$};
    \node [right] at (1.2,.5){24};
    \node [right] at (1.2, 1.5){2};
    \node [right] at (1.2,2.5){24};
\end{tikzpicture}
\caption{Subgroup lattice for $G_W$ when $W =x^2+y^3+z^8+w^{24}$ }
\label{group_lattice}
\end{figure}
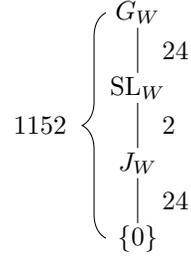

The non--symplectic automorphism $\sigma_3$ of order 3 is 
\[
\pp xyzw \mapsto \pp x{\zeta_3y}zw.
\]
Written additively, this is $\sigma_3=(0,\tfrac 13,0,0)$. Now we have two K3 surfaces to consider, which are BHCR mirrors, namely $X_{W,J_W}$ and $X_{W,\SLgp{W}}$. We need to show that $(X_{W,J_W}, S(\sigma_3))$ and $(X_{W,\SLgp{W}},S(\sigma_3))$ are LPK3 mirrors. To do this we consider each one separately. 

\begin{itemize}
\item $(W,J_W)$: 

We will first describe the surface $Y_W=\set{W=0}\subset\PP(12,8,3,1)$ more explicitly, and then describe its resolution $X_{W,J_W}$. Since $Y_W$ is quasismooth, the singular points of surface $Y_W$ only occur at singular points of $\PP(12,8,3,1)$. To find them, we examine the action of $\Cstar$ on $\C^4$, which is given by 
\[
\lambda\cdot (x,y,z,w)=(\lambda^{12}x, \lambda^8 y, \lambda^3 z, \lambda w).
\] 
As mentioned previously, the singularities in question will be occur at points $\pp xyzw\in Y_W$ for which two or three of the coordinates vanish and for which the action of $\Cstar$ has non-trivial isotropy. 

From the action, we see that the point $\pp1{-1}00 \in Y_W$ is a point with $\Zfin{4}$ isotropy, arising from $\lambda=i$.
Since the order of the isotropy is 4, $Y_W$ has an $A_3$ singularity at this point.

Moreover, the points $\pp {\pm i}010$ are 2 points with $\mathbb Z/3\mathbb Z$ isotropy arising from $\lambda=\zeta_3$ and so they give $2A_2$ singularities in $Y_W$. Resolving these singularities, we obtain $X_W$. We have diagrammed the resolution in Figure \ref{xw}. 

\begin{figure}[ht]
\centering
\begin{tikzpicture}[xscale=.6,yscale=.5]
\draw (0,6)--(6.5,6);
\node [left] at (0,6){$y=0$};
\draw (0,2)--(6.5,2);
\node [left] at (0,2){$w=0$};

\draw [thick] (2.2,1.8)--(0.8,4.2);
\draw [thick] (0.8,3.8)--(2.2,6.2);

\draw [thick] (4.2,1.8)--(2.8,4.2);
\draw [thick] (2.8,3.8)--(4.2,6.2);

\draw [thick] (4.8,2.2)--(6.2,.8);
\draw [thick] (6.2,1.2)--(4.8,-.2);
\draw [thick] (4.8,.2)--(6.2,-1.2);
\end{tikzpicture}
\caption{Curve configurations on $X_W$ in Example \ref{ex1}}
\label{xw}
\end{figure}
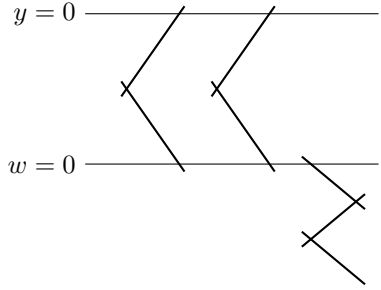

All that remains is to calculate the invariant lattice of the non-symplectic automorphism $\sigma_3$. This is determined by $r$ and $a$, which in turn can be calculated by considering the fixed locus of the action of $\sigma_3$ on $X_W$. 

The curve $y=0$ is obviously fixed. Furthermore, by the calculation 
\begin{align*}
\pp x{\zeta_3y}zw &= \pp{\zeta_3^{12}x}{\zeta_3^9 y}{\zeta_3^3z}{\zeta_3 w}\\
    &=\pp xyz{\zeta_3w},
\end{align*}
it follows that an alternate description of $\sigma_3$ is 
\[
\pp xyzw \mapsto \pp xyz{\zeta_3w}.
\]
From this description, we see that $w=0$ is also fixed. 

To describe this computation more generally, we want to find other descriptions of $\sigma_3$. Since we are working on the quotient by $\Cstar$, written additively, we have 
\[
(\gamma_1,\gamma_2,\gamma_3,\gamma_4)\sim (0,\tfrac 13, 0,0)\imod{\Z}
\]
if and only if there exists $q\in \Q/\Z$ such that 
\[
(0,\tfrac 13, 0,0)+(12q,8q,3q,q)\equiv (\gamma_1,\gamma_2,\gamma_3,\gamma_4)\imod{\Z}.
\]
In general, we want to consider representatives of $\sigma_3$ which fix one or more coordinates. There are only finitely many $q$ giving us such a representative. 

By Lemma \ref{eq_genus}, the curves $y=0$ and $w=0$ have genus $3$ and $0$, respectively.


Finally, we can compute $n$ and $k$ explicitly from the action of $\sigma_3$ by examining Table \ref{tab_lattices3}. Since $w=0$ is fixed by $\sigma_3$, each of the exceptional divisors from the resolution of singularities is invariant (though perhaps not fixed pointwise). Nothing else is fixed. Since the fixed curves must be disjoint, by examining Table \ref{tab_lattices3} we find that the only possible configuration of fixed points and fixed curves has one curve of genus 3, three isolated points and 2 rational curves, i.e. $(g,n,k)=(3,3,2)$. From there one can easily compute $(r,a)$ using Theorem \ref{summary_invariants}.

Referring to Table \ref{tab_lattices3} in the Appendix, we see that the automorphism $\sigma_3$ has invariant lattice $S(\sigma_3) = U\oplus E_6$, $S(\sigma_3)^\perp=T(\sigma_3)=U\oplus U\oplus A_2\oplus E_8$, and so $\mirror{S(\sigma_3)}=U\oplus A_2\oplus E_8$. 

\medskip

\item $(W,\SLgp{W})$:

First we provide a description of $X_W/\widetilde{\SL}_W$. To do so, we find the isotropy points of the action of $\widetilde{\SL}_W = \SL_{W}/J_{W}$ on $X_W$. Here $\widetilde{\SL}_W\cong\mathbb Z/2\mathbb Z$ and hence, from \citep[Section 5]{nikulin_finite}, we know that the action of $\widetilde{\SL}_W$ on the resolution of $X_{W}$ has eight fixed points, each resulting in an $A_1$ singularity on the quotient (see also \citep{gang}).

In general, for $J_W\subset G\subset \SL_W$, we find these fixed points by finding representatives $\gamma_1,\dots, \gamma_k$ for generators of $G\Cstar/\Cstar\cong G/J_W$. Then, similar to the previous calculation with $\sigma_3$, we find other representatives $\gamma_i'\sim \gamma_i$ such that $\gamma_i'$ fixes one or more coordinates. In this case we need only one generator of $\SL_W/J_W$, which we may take as $\gamma=(\tfrac 12,0,0,\tfrac 12)$. The other representatives we need to consider are $(\tfrac 12,0,\tfrac 12, 0)$ and $(0,0,\tfrac 38, \tfrac 58)$. 

From the first representative, we find that the point $\pp0{-1}10$ is fixed. From the second, the points $[0:\zeta_6^j:0:1]$, $j=1,3,5$ (where $\zeta_6$ is a primitive $6$-th root of unity) are fixed. From the third representative, we find that the remaining four isotropy points lie on the resolution of the $A_3$ singularity. To complete the analysis, the configurations of curves from the $A_2$ singularities, which were obtained by blowing up the points $\pp{\pm i}010$, are permuted by $\gamma$; hence these two configurations become identified on the quotient $X/\widetilde{\SL}_W$.

Note that $\widetilde{\SL}_W$ preserves the curves $y=0$ and $w=0$, so that the diagram for the resolution $X_{W,\SL_W}$ of the quotient $X_W/\widetilde{\SL}_W$ looks as in Figure \ref{xwsl}. There we have labeled the images of $y=0$ and $w=0$ somewhat abusively as $y=0$ and $w=0$, and we have represented the new exceptional curves with dashed lines. Note that there are eight of them. 

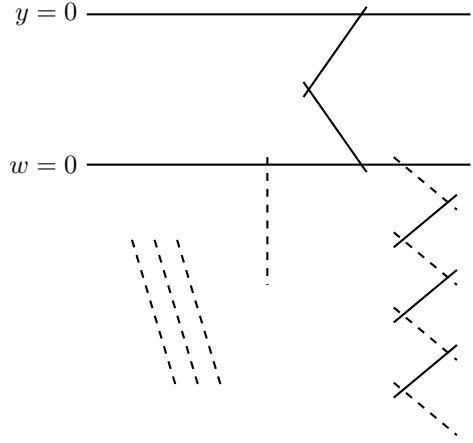
\begin{figure}[ht]
\centering\begin{tikzpicture}[xscale=.6,yscale=.5]
\draw [thick](-2,6)--(6.5,6);
\node [left] at (-2,6){$y=0$};
\draw [thick](-2,2)--(6.5,2);
\node [left] at (-2,2){$w=0$};

\draw [thick] (4.2,1.8)--(2.8,4.2);
\draw [thick] (2.8,3.8)--(4.2,6.2);

\draw [thick, dashed] (0,0)--(1,-4);
\draw [thick, dashed] (-.5,0)--(.5,-4);
\draw [thick, dashed] (-1,0)--(0,-4);

\draw [thick, dashed] (2,2.2)--(2,-1.2);

\draw [thick,dashed] (4.8,2.2)--(6.2,.8);
\draw [thick] (6.2,1.2)--(4.8,-.2);
\draw [thick,dashed] (4.8,.2)--(6.2,-1.2);
\draw [thick] (6.2,-.8)--(4.8,-2.2);
\draw [thick,dashed] (4.8,-1.8)--(6.2,-3.2);
\draw [thick] (6.2,-2.8)--(4.8,-4.2);
\draw [thick,dashed] (4.8,-3.8)--(6.2,-5.2);

\end{tikzpicture}

\caption{Curve configurations on $X_{W,\SL_W}$ in Example \ref{ex1} (dashed curves arise from resolution of $X_W/\widetilde{\SL}_W$)}

\label{xwsl}
\end{figure}

We can use the Reimann-Hurwitz formula to calculate the genus of (the image of) the curve $y=0$ in $X_{W,\SL_W}$. Indeed, we have a degree 2 cover with no ramification. Thus 
\begin{align*}
2-2g_{old} &= 2 (2-2g_{new}) - \sum_{\rho \in S\p} (e_{\rho} -1)\\
2 - 2\cdot 3 &= 2(2 - 2g_{new}) - 0\\
g_{new} &= 2
\end{align*}
In the same way, we find that the genus of the curve $w=0$ in the resolved surface is 0. Indeed, we have a degree 2 cover by a curve of genus 0, with four points of ramification index 2.  Thus
\begin{align*}
2-2g_{old} &= 2 (2-2g_{new}) - \sum_{\rho \in S\p} (e_{\rho} -1)\\
2 - 2\cdot 0 &= 2(2 - 2g_{new}) - 4\\
g_{new} &= 0
\end{align*}


Finally, when considering the action of $\sigma_3$ on $X_{W,\SL_W}$, $y=0$ and $w=0$ are invariant. Furthermore, three of the exceptional curves are permuted by $\sigma_3$ and the other ten are invariant. In order to compute $n$ and $k$, we again consult Table \ref{tab_lattices3}. We find that the only possible configuration of fixed points and fixed curves has one curve of genus 2, five isolated points and four rational curves, i.e. $(g,n,k)=(2,5,4)$. Thus $(r,a)=(12,1)$. By Table \ref{tab_lattices3}, the invariant lattice is then $S(\sigma_3)= U\oplus E_8\oplus A_2$. 
\end{itemize}

Since the mirror lattice in the first case matches the fixed lattice in the second, this verifies Theorem \ref{main_thm} in this example.

\end{ex}


\begin{ex}\label{ex2}
We now describe an example in which $W\neq W^T$.  The calculation of the invariants $(r,a)$ in the various cases proceeds exactly as in Example \ref{ex1}.  We summarize the results below.

To start, consider
\[
W = x^2+y^5+z^5+xw^5,
\]
which corresponds to No.\ 6c in Table \ref{tab_eq5}.  This is quasihomogenous with weight system $(5,2,2,1;10)$ and has the non-symplectic automorphism $\sigma_5 = (0,\frac15,0,0)$.  Its transpose is
\[
W^T = x^2w+y^5+z^5+w^5,
\]
which corresponds to No.\ 21a in Table \ref{tab_eq5} and has weight system $(2,1,1,1;5)$.  Thus $j_W = (\frac12,\frac15,\frac15,\frac1{10})$ has order 10 and $j_{W^T} = (\frac25,\frac15,\frac15,\frac15)$ has order 5.  From this we conclude that
\[
|G_W/SL_W| = |(\SL_W)^T/(G_W)^T| = |J_{W^T}/\set{0}| = 5.
\]
Since $|G_W| = \det(A_W) = 250$ and $|J_W|=10$, it follows that one has the diagram of subgroups of $G_W$ to be as in Figure \ref{group_lattice_2}.
\begin{figure}[h]
\centering
\begin{tikzpicture}
[xscale = 1, yscale = 1]
    \draw (1,.2) --(1,.8) ;
    \draw (1,1.2) --(1,1.8) ;
    \draw (1,2.2) --(1,2.8) ;
    
    \node at (1,0) {$\triv$};
    \node at (1,1) {$J_{W}$};
    
    \node at (1,2) {$\SL_{W}$};
    \node at (1,3) {$G_W$};
    \draw [decorate,decoration={brace,amplitude=9pt},xshift=-12pt,yshift=0pt] (1,0) -- (1,3)node [black,midway,xshift=-25pt] { $250$};
    \node [right] at (1.2,.5){10};
    \node [right] at (1.2, 1.5){5};
    \node [right] at (1.2,2.5){5};
\end{tikzpicture}
\caption{Subgroup lattice for $G_W$ when $W=x^2+y^5+z^5+xw^5$}
\label{group_lattice_2}
\end{figure}
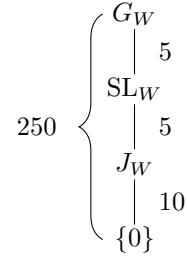
In particular, the index of $J_W$ in $\SL_W$ is 5, and so the only subgroups $J_W \subset G\subset \SL_W$ are $G=J_W$ and $G=\SL_W$.  

\begin{itemize}

\item $(W,J_W)$:

One finds that $Y_W$ contains $5A_1$ singularities (comprising the set $\set{x=w=0}$) that resolve to yield five exceptional curves on $X_W$.  These five curves are permuted by the action of $\sigma_5=(0,\frac15,0,0)$.  Also, $\sigma_5$ fixes the curve $y=0$, which has genus $g=2$ by Lemma \ref{eq_genus}. Finally, an equivalent form of $\sigma_5$ is $(0,0,\frac 45,\frac 25)$, which is seen to fix the point corresponding to $z=w=0$. Hence we have $(g,n,k)=(2,1,0)$ and by Theorem \ref{summary_invariants}, the invariants of $S(\sigma_5)$ are $(r,a)=(2,1)$.  It follows from Table \ref{tab_lattices5} that $S(\sigma_5) = H_5$ and $\mirror{S(\sigma_5)} = H_5\oplus E_8^2$.

\medskip

\item $(W,\SL_W)$:

One element of $\SL_W\setminus J_W$ is $\tau = (0,\frac45,\frac15,0)$, which also has equivalent forms $(0,\frac35,0,\frac25)$ and $(0,0,\frac25,\frac35)$.  Thus the symplectic action of $\widetilde{\SL}_W$ on $X_W$ fixes $\set{y=z=0}$, $\set{y=w=0}$, and $\set{z=w=0}$.  This amounts to four isolated fixed points, which yield $4A_4$ singularities on the quotient $X_W/\widetilde{\SL}_W$ that resolve to give 16 exceptional curves on $X_{W,\SL_W}$.  All of these exceptional curves are invariant under the action of $\sigma_5$ on $X_{W,\SL_W}$, as is the image of the exceptional curve from $X_W$. Furthermore, $\sigma_5$ fixes (the images of) $y=0$, $z=0$ and $w=0$. By Riemann-Hurwitz, one calculates the genus of these curves on $X_{W,\SL_W}$ to be $0$, and there are no fixed curves of higher genus.  Again from table \ref{tab_lattices3}, the only possible configuration of fixed points and curves has three rational curves and 13 fixed points, giving us $(g,n,k)=(0,13,2)$ and so $(r,a)=(18,1)$. This gives $S(\sigma_5) = H_5\oplus E_8^2$ and $\mirror{S(\sigma_5)} = H_5$. 
\begin{rem}
In computing $k$ recall that the fixed curve $C$ from \eqref{eq_fixed_locus} may have genus 0, as happens in this case. There are indeed three fixed rational curves, but one of them is $C$, and hence $k=2$ (see the remark following Theorem \ref{summary_invariants}). 
\end{rem}

\end{itemize}

Next we make the same types of calculations for the BHCR mirrors, which correspond to No.\ 21a in Table \ref{tab_eq5}.  The subgroup lattice of $G_{W^T}$ is given in Figure \ref{group_lattice_3}.  (Note that this is essentially the result of applying the transpose operation to Figure \ref{group_lattice_2}.)
\begin{figure}[h]
\centering
\begin{tikzpicture}
[xscale = 1, yscale = 1]
    \draw (1,.2) --(1,.8) ;
    \draw (1,1.2) --(1,1.8) ;
    \draw (1,2.2) --(1,2.8) ;
    
    \node at (1,0) {$\triv$};
    \node at (1,1) {$J_{W^T}$};
    \node at (1,2) {$\SL_{W^T}$};
    \node at (1,3) {$G_{W^T}$};
    \draw [decorate,decoration={brace,amplitude=9pt},xshift=-12pt,yshift=0pt] (1,0) -- (1,3)node [black,midway,xshift=-25pt] { $250$};
    \node [right] at (1.2,.5){5};
    \node [right] at (1.2, 1.5){5};
    \node [right] at (1.2,2.5){10};
\end{tikzpicture}
\caption{Subgroup lattice for $G_{W^T}$ when $W^T=x^2w+y^5+z^5+w^5$}
\label{group_lattice_3}
\end{figure}
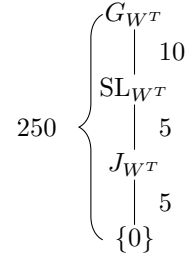
The groups $J_{W^T} \subset G\subset \SL_{W^T}$ to consider are $G=J_{W^T}$ and $G=\SL_{W^T}$.  Below we denote the non-symplectic automorphism on $X_{W^T,G}$ by $\sigma_5^T=(0,\frac15,0,0)$.

\begin{itemize}

\item $(W^T,J_{W^T})$:

The hypersurface $Y_{W^T}$ has one $A_1$ singularity, which lies on the curve $y=0$. Its resolution gives an exceptional curve on $X_{W^T}$, and this exceptional curve is left invariant by $\sigma_5^T$. The only fixed curve of $\sigma_5^T$ is $y=0$, which has genus $g=2$. As fixed curves must be disjoint, the exceptional curve is not fixed, but must contain one other fixed point. Thus $(g,n,k)=(2,1,0)$. Therefore $(r,a)=(2,1)$ and $S(\sigma_5^T) = H_5$, $\mirror{S(\sigma_5^T)} = H_5\oplus E_8^2$.

\medskip

\item $(W^T,\SL_{W^T})$:

An element of $\SL_{W^T}\setminus J_{W^T}$ is $(0,\frac15,\frac45,0)$.  It fixes exactly four points on $X_{W^T}$, and hence the quotient $X_{W^T}/\widetilde{\SL}_{W^T}$ has $4A_4$ singularities, whose resolution yields 16 exceptional curves on $X_{W^T,\SL_{W^T}}$.  All of these are invariant under the action of $\sigma_5^T$, as is the exceptional curve coming from $X_{W^T}$.  The fixed curve $y=0$ has genus $g=0$ and the fixed curve $z=0$ has genus 0, as well. From the equivalence $\sigma_5=(\frac 45,0,0,\frac 25)$ we get a fixed point corresponding to $x=w=0$. Finally, the 16 exceptional curves contribute one fixed curve and 12 fixed points, giving us in total $(g,n,k)=(0,13,2)$. Hence we obtain $(r,a)=(18,1)$ and $S(\sigma_5^T)=H_5\oplus E_8^2$, $\mirror{S(\sigma_5^T)} = H_5$.
\begin{rem}
Again here we have the number of fixed rational curves equal to $k+1$ (see the Remark following Theorem \ref{summary_invariants}).
\end{rem}

\end{itemize}

Having completed these calculations, the theorem for this case is verified by comparison.  Looking at the calculations for the BHCR mirror pair $X_{W}$, $X_{W^T,\SL_{W^T}}$, one sees that $\mirror{S(\sigma_5)} = S(\sigma_5^T)$; thus the lattice polarized K3 surfaces $(X_{W},S(\sigma_5))$ and $(X_{W^T,\SL_{W^T}},S(\sigma_5^T))$ are LPK3 mirrors.  A similar statement holds for the BHCR mirror pair $X_{W,\SL_{W}}$, $X_{W^T}$.

\end{ex}

\begin{ex}\label{ex3}
The data in Table 1 is insufficient for verifying Theorem \ref{main_thm} for entry No.\ 3a, where
\[
W = x^3+y^3+z^6+w^6
\]
and $\sigma_3 = (\frac13,0,0,0)$.  We include here the necessary further details. One may show that $\SL_W/J_W \simeq (\Z/3\Z)^2$, so that there are four distinct subgroups $J_W\subset G\subset \SL_W$ for which $|G/J_W|=3$. Using the same order as in Table \ref{tab_eq3}, these four subgroups may be described as $G_i = \inn{g_i}$, $i=1,\ldots,4$, where
\[
g_1 = \left(\frac23,\frac13,0,0\right), \qquad g_2 = \left(\frac13,\frac13,\frac13,0\right), \qquad g_3 = \left(\frac23,0,\frac13,0\right), \qquad g_4 = \left(0,\frac23,\frac13,0\right)
\]
Since $W=W^T$, we have $G_W=G_{W^T}$.  By (\ref{dualG_def}) and the calculation $g_1 A_W g_2^T \in\Z$, we find $G_1^T=G_2$ and $G_2^T=G_1$.  On the other hand, from the calculations $g_3 A_W g_3^T \in \Z$ and $g_4 A_W g_4^T \in \Z$, we conclude that $G_3^T=G_3$ and $G_4^T=G_4$.

Using the methods described the previous examples, one calculates for the surface $X_{W,G}$ that the invariants $(r,a)$ of $S(\sigma_3)$ are those given in the following table:

\begin{center}
\begin{tabular}{c|c|c}
$G$ & $r$ & $a$ \\
\hline
$G_1$ & 14 &4 \\
$G_2$ &  6& 4\\
$G_3$ & 10 & 4 \\
$G_4$ & 10 & 4 \\
\end{tabular}
\end{center}

\noindent In each case, we see that if $(r,a)$ are the invariants of $S(\sigma_3)$ for $X_{W,G}$, then the corresponding invariants for $X_{W,G^T}$ are $(20-r,a)$. This proves the theorem when $G=G_i,i=1\ldots 4$.

\end{ex}

\section{Tables}\label{tables}
This section contains tables listing all $p$-cyclic K3 surfaces for $p=3,5,7,13$.  We compute the invariants $(r,a)$ of the fixed lattice $S(\sigma_p)$ for each surface using the methods illustrated in Section \ref{ex_sec}. 

Numbering of the weight systems follows the numbering in \citep{yonemura}.

\begin{alphafootnotes}


\begin{longtable}{@{}lllcccc@{}}\toprule
\multicolumn{7}{c}{Table for $p=3$} \\ 

No.    &Weights	&Polynomial	&$\left|\SLgp{W}/J_{W}\right|$	&$\left|G/J_{W}\right|$	&$(r,a)$	&BHCR Dual	\\ \midrule \endfirsthead
\toprule
No.    &Weights	&Polynomial	&$\left|\SLgp{W}/J_{W}\right|$	&$\left|G/J_{W}\right|$	&$(r,a)$	&BHCR Dual	\\ \midrule \endhead

2a	&(4,3,3,2;12)	&$x^3+y^3z+yz^3+w^6$	&1	&1	&(10,4)	&2a\\
2b	&(4,3,3,2;12)	&$x^3+y^4+z^4+w^6$	&2	&2	&(10,4)	&2b	\\*
	&	&	&	&1	&(10,4)	&	\\
2c	&(4,3,3,2;12)	&$x^3+y^3z+z^4+w^6$	&3	&3	&(18,2)	&3c	\\*
	&	&	&	&1	&(10,4)	&	\\
3a\footnote{There are four subgroups of order 3. We do not differentiate them here, other than listing $(r,a)$. See Example \ref{ex3} in Section \ref{ex_sec} for disambiguation.}
&(2,2,1,1;6)	&$x^3+y^3+z^6+w^6$	&9	&9	&(18,2)	&3a	\\*
	&	&	&	&3	&(14,4)	&	\\*
	&	&	&	&3	&(6,4)	&	\\*
	&	&	&	&3	&(10,4)	&	\\*
	&	&	&	&3	&(10,4)	&	\\*
	&	&	&	&1	&(2,2)	&	\\
3b	&(2,2,1,1;6)	&$x^3+y^3+z^5w+w^6$	&3	&3	&(14,4)	&17a	\\*
	&	&	&	&1	&(2,2)	&	\\
3c	&(2,2,1,1;6)	&$x^3+y^3+yz^4+w^6$	&3	&3	&(10,4)	&2c	\\*
	&	&	&	&1	&(2,2)	&	\\
3d	&(2,2,1,1;6)	&$x^3+y^3+z^5w+zw^5$	&6	&6	&(18,2)	&3d	\\*
	&	&	&	&3	&(14,4)	&	\\*
	&	&	&	&2	&(6,4)	&	\\*
	&	&	&	&1	&(2,2)	&	\\
3e	&(2,2,1,1;6)	&$x^3+y^3+yz^4+zw^5$	&2	&2	&(6,4)	&15a	\\*
	&	&	&	&1	&(2,2)	&	\\
4a	&(4,4,3,1;12)	&$x^3+y^3+z^4+w^{12}$	&3	&3	&(16,3)	&4a	\\*
	&	&	&	&1	&(4,3)	&	\\
4b	&(4,4,3,1;12)	&$x^3+y^3+z^4+zw^9$	&3	&3	&(16,3)	&18a	\\*
	&	&	&	&1	&(4,3)	&	\\
4c	&(4,4,3,1;12)	&$x^3+y^3+z^4+yw^8$	&1	&1	&(4,3)	&16b	\\*
10a	&(6,4,1,1;12)	&$x^2+y^3+z^{11}w+w^{12}$	&1	&1	&(2,0)	&46\\
10b	&(6,4,1,1;12)	&$x^2+xw^6+y^3+z^{11}w$	&1	&1	&(2,0)	&65	\\
10c	&(6,4,1,1;12)	&$x^2+y^3+z^{11}w+zw^{11}$	&5	&5	&(18,0)	&10c	\\*
	&	&	&	&1	&(2,0)	&	\\
10d	&(6,4,1,1;12)	&$x^2+y^3+z^{12}+w^{12}$	&6	&6	&(18,0)	&10d	\\*
	&	&	&	&3	&(14,2)	&	\\*
	&	&	&	&2	&(6,2)	&	\\*
	&	&	&	&1	&(2,0)	&	\\
10e	&(6,4,1,1;12)	&$x^2+y^3+xz^6+w^{12}$	&3	&3	&(14,2)	&24b	\\*
	&	&	&	&1	&(2,0)	&	\\
11a	&(15,10,3,2;30)	&$x^2+y^3+z^{10}+w^{15}$	&1	&1	&(10,2)	&11a\\
11b	&(15,10,3,2;30)	&$x^2+y^3+xz^5+w^{15}$	&1	&1	&(10,2)	&22b\\
12a	&(9,6,2,1;18)	&$x^2+y^3+z^9+zw^{16}$	&1	&1	&(4,1)	&48a\\
12b	&(9,6,2,1;18)	&$x^2+y^3+z^9+xw^9$	&3	&3	&(16,1)	&25d	\\*
	&	&	&	&1	&(4,1)	&	\\
12c	&(9,6,2,1;18)	&$x^2+y^3+z^9+w^{18}$	&3	&3	&(16,1)	&12c	\\*
	&	&	&	&1	&(4,1)	&	\\
13a	&(12,8,3,1;24)	&$x^2+y^3+xz^4+w^{24}$	&1	&1	&(8,1)	&20a\\
13b	&(12,8,3,1;24)	&$x^2+y^3+xz^4+zw^{21}$	&1	&1	&(8,1)	&59a\\
13c	&(12,8,3,1;24)	&$x^2+y^3+z^8+xw^{12}$	&1	&1	&(8,1)	&27	\\
13d	&(12,8,3,1;24)	&$x^2+y^3+z^8+w^{24}$	&2	&2	&(12,1)	&13d\\*
	&	&	&	&1	&(8,1)	&	\\
13e	&(12,8,3,1;24)	&$x^2+y^3+z^8+zw^{21}$	&1	&1	&(8,1)	&49	\\
14a	&(21,14,6,1;42)	&$x^2+y^3+z^7+w^{42}$	&1	&1	&(10,0)	&14a\\
14b	&(21,14,6,1;42)	&$x^2+y^3+z^7+xw^{21}$	&1	&1	&(10,0)	&28b\\
14c	&(21,14,6,1;42)	&$x^2+y^3+z^7+zw^{36}$	&1	&1	&(10,0)	&51b\\
15a	&(5,4,3,3;15)	&$x^3+y^3z+z^4w+w^5$	&2	&2	&(18,2)	&3e	\\*
	&	&	&	&1	&(14,4)	&	\\
15b	&(5,4,3,3;15)	&$x^3+y^3z+z^5+w^5$	&1	&1	&(14,4)	&17b	\\
16a	&(8,7,6,3;24)	&$x^3+y^3w+z^4+zw^6$	&1	&1	&(16,3)	&18b\\
16b	&(8,7,6,3;24)	&$x^3+y^3w+z^4+w^8$	&1	&1	&(16,3)	&4c	\\
17a	&(5,5,3,2;15)	&$x^3+y^3+z^5+zw^6$	&3	&3	&(18,2)	&3b	\\*
	&	&	&	&1	&(6,4)	&	\\
17b	&(5,5,3,2;15)	&$x^3+y^3+z^5+yw^5$	&1	&1	&(6,4)	&15b\\
18a	&(3,3,2,1;9)	&$x^3+y^3+z^4w+w^9$	&3	&3	&(16,3)	&4b	\\*
	&	&	&	&1	&(4,3)	&	\\
18b	&(3,3,2,1;9)	&$x^3+y^3+z^4w+yw^6$	&1	&1	&(4,3)	&16a\\
18c	&(3,3,2,1;9)	&$x^3+y^3+xz^3+zw^7$	&1	&1	&(4,3)	&54	\\
18d	&(3,3,2,1;9)	&$x^3+y^3+z^4w+zw^7$	&3	&3	&(16,3)	&18d\\*
	&	&	&	&1	&(4,3)	&	\\
18e	&(3,3,2,1;9)	&$x^3+y^3+z^3x+w^9$	&3	&3	&(16,3)	&18e\\*
	&	&	&	&1	&(4,3)	&	\\
20a	&(9,8,6,1;24)	&$x^2z+y^3+z^4+w^{24}$	&1	&1	&(12,1)	&13a	\\
20b	&(9,8,6,1;24)	&$x^2z+y^3+z^4+xw^{15}$	&1	&1	&(12,1)	&72a	\\
22a	&(6,5,3,1;15)	&$x^2z+y^3+xz^3+w^{15}$	&1	&1	&(10,2)	&22a	\\
22b	&(6,5,3,1;15)	&$x^2z+y^3+z^5+w^{15}$	&1	&1	&(10,2)	&11b	\\
22c	&(6,5,3,1;15)	&$x^2z+y^3+z^5+xw^9$	&2	&2	&(16,1)	&25c	\\*
	&	&	&	&1	&(10,2)	&	\\
24a	&(5,4,2,1;12)	&$x^2z+y^3+z^6+xw^7$	&1	&1	&(6,2)	&67	\\
24b	&(5,4,2,1;12)	&$x^2z+y^3+z^6+w^{12}$	&3	&3	&(18,0)	&10e\\*
	&	&	&	&1	&(6,2)	&	\\
25a	&(4,3,1,1;9)	&$x^2w+y^3+xz^5+zw^8$	&1	&1	&(4,1)	&88	\\
25b	&(4,3,1,1;9)	&$x^2z+y^3+z^8w+w^9$	&1	&1	&(4,1)	&48b\\
25c	&(4,3,1,1;9)	&$x^2w+y^3+xz^5+w^9$	&2	&2	&(10,2)	&22c\\*
	&	&	&	&1	&(4,1)	&	\\
25d	&(4,3,1,1;9)	&$x^2w+y^3+z^9+w^9$	&3	&3	&(16,1)	&12b	\\*
	&	&	&	&1	&(4,1)	&	\\
25e	&(4,3,1,1;9)	&$x^2w+y^3+z^9+xw^5$	&3	&3	&(16,1)	&25e\\*
	&	&	&	&1	&(4,1)	&	\\
27	&(11,8,3,2;24)	&$x^2w+y^3+z^8+w^{12}$	&1	&1	&(12,1)	&13c	\\
28a	&(10,7,3,1;21)	&$x^2w+y^3+z^7+xw^{11}$	&1	&1	&(10,0)	&28a	\\
28b	&(10,7,3,1;21)	&$x^2w+y^3+z^7+w^{21}$	&1	&1	&(10,0)	&14b	\\
28c	&(10,7,3,1;21)	&$x^2w+y^3+z^7+zw^{18}$	&1	&1	&(10,0)	&51c	\\
46	&(33,22,6,5;66)	&$x^2+y^3+z^{11}+zw^{12}$	&1	&1	&(18,0)	&10a\\
48a	&(24,16,5,3;48)	&$x^2+y^3+z^9w+w^{16}$	&1	&1	&(16,1)	&12a	\\
48b	&(24,16,5,3;48)	&$x^2+y^3+z^9w+xw^8$	&1	&1	&(16,1)	&25b	\\
49	&(21,14,5,2;42)	&$x^2+y^3+z^8w+w^{21}$	&1	&1	&(12,1)	&13e	\\
51a	&(18,12,5,1;36)	&$x^2+y^3+z^7w+zw^{31}$	&1	&1	&(10,0)	&51a	\\
51b	&(18,12,5,1;36)	&$x^2+y^3+z^7w+w^{36}$	&1	&1	&(10,0)	&14c	\\
51c	&(18,12,5,1;36)	&$x^2+y^3+z^7w+xw^{18}$	&1	&1	&(10,0)	&28c	\\
54	&(7,6,5,3;21)	&$x^3+y^3w+yz^3+w^7$	&1	&1	&(16,3)	&18c	\\
59a	&(8,7,5,1;21)	&$x^2z+y^3+z^4w+w^{21}$	&1	&1	&(12,1)	&13b	\\
59b	&(8,7,5,1;21)	&$x^2z+y^3+z^4w+xw^{13}$	&1	&1	&(12,1)	&72b	\\
65	&(14,11,5,3;33)	&$x^2z+y^3+z^6w+w^{11}$	&1	&1	&(18,0)	&10b	\\
67	&(9,7,3,2;21)	&$x^2z+y^3+z^7+xw^6$	&1	&1	&(14,2)	&24a	\\
72a	&(7,5,2,1;15)	&$x^2w+y^3+xz^4+w^{15}$	&1	&1	&(8,1)	&20b	\\
72b	&(7,5,2,1;15)	&$x^2w+y^3+xz^4+zw^{13}$	&1	&1	&(8,1)	&59b	\\
88	&(11,9,5,2;27)	&$x^2z+y^3+z^5w+xw^8$	&1	&1	&(16,1)	&25a	\\
\bottomrule

\caption[Table for $p=3$]{Calculations for $p=3$}\label{tab_eq3}
\end{longtable}

\begin{table}[h]
\centering

\begin{tabular}{@{}lllcccc@{}}\toprule

\multicolumn{7}{c}{Table for $p=5$} \\

No.    &Weights    &Polynomial	&$\left|\SLgp{W}/J_{W}\right|$	&$\left|G/J_{W}\right|$	&$(r,a)$	&BHCR Dual	\\ \midrule
6a	&(5,2,2,1;10)	&$x^2+y^5+z^5+w^{10}$	&5	&5	&(18,1)	&6a	\\
	&	&	&	&1	&(2,1)	&	\\
6b	&(5,2,2,1;10)	&$x^2+y^5+z^5+zw^8$	&1	&1	&(2,1)	&30a	\\
6c	&(5,2,2,1;10)	&$x^2+y^5+z^5+xw^5$	&5	&5	&(18,1)	&21a	\\
	&	&	&	&1	&(2,1)	&	\\
9a	&(10,5,4,1;20)	&$x^2+y^4+z^5+w^{20}$	&2	&2	&(10,1)	&9a	\\
	&	&	&	&1	&(10,1)	&	\\
9b	&(10,5,4,1;20)	&$x^2+y^4+z^5+yw^{15}$	&1	&1	&(10,1)	&34	\\
9c	&(10,5,4,1;20)	&$x^2+y^4+z^5+xw^{10}$	&1	&1	&(10,1)	&26	\\
9d	&(10,5,4,1;20)	&$x^2+xy^2+z^5+w^{20}$	&1	&1	&(10,1)	&9d	\\
9e	&(10,5,4,1;20)	&$x^2+xy^2+z^5+yw^{15}$	&1	&1	&(10,1)	&71a\\
15b	&(5,4,3,3;15)	&$x^3+y^3z+z^5+w^5$	&1	&1	&(6,2)	&17b	\\
17b	&(5,5,3,2;15)	&$x^3+y^3+z^5+yw^5$	&1	&1	&(14,2)	&15b	\\
17c	&(5,5,3,2;15)	&$x^2y+y^3+z^5+xw^5$	&2	&2	&(18,1)	&21d	\\
	&	&	&	&1	&(14,2)	&	\\
21a	&(2,1,1,1;5)	&$x^2w+y^5+z^5+w^5$	&5	&5	&(18,1)	&6c	\\
	&	&	&	&1	&(2,1)	&	\\
21b	&(2,1,1,1;5)	&$x^2y+y^4z+z^5+w^5$	&1	&1	&(2,1)	&30b\\
21c	&(2,1,1,1;5)	&$x^2z+y^5+z^4w+xw^3$	&1	&1	&(2,1)	&86	\\
21d	&(2,1,1,1;5)	&$x^2z+xy^3+z^5+w^5$	&2	&2	&(6,2)	&17c	\\
	&	&	&	&1	&(2,1)	&	\\
21e	&(2,1,1,1;5)	&$x^2w+y^5+z^5+xw^3$	&5	&5	&(18,1)	&21e	\\
	&	&	&	&1	&(2,1)	&	\\
26	&(9,5,4,2;20)	&$x^2w+y^4+z^5+w^{10}$	&1	&1	&(10,1)	&9c\\
30a	&(20,8,7,5;40)	&$x^2+y^5+z^5w+w^8$	&1	&1	&(18,1)	&6b	\\
30b	&(20,8,7,5;40)	&$x^2+y^5+z^5w+xw^4$	&1	&1	&(18,1)	&21b\\
34	&(15,7,6,2;30)	&$x^2+y^4w+z^5+w^{15}$	&1	&1	&(10,1)	&9b	\\
71a	&(4,7,3,1;15)	&$x^2y+y^2w+z^5+w^{15}$	&1	&1	&(10,1)	&9e	\\
71b	&(4,7,3,1;15)	&$x^2w+xy^2+z^5+yw^{11}$	&1	&1	&(10,1)	&71b\\
86	&(9,7,5,4;25)	&$x^2y+y^3w+z^5+xw^4$	&1	&1	&(18,1)	&21c\\
\bottomrule
\end{tabular}
\caption[Table for $p=5$]{Calculations for $p=5$}\label{tab_eq5}

\end{table}
\begin{table}
\centering

\begin{tabular}{@{}lllcccc@{}}\toprule
\multicolumn{7}{c}{Table for $p=7$} \\

No.    &Weights    &Polynomial	&$\left|\SLgp{W}/J_{W}\right|$	&$\left|G/J_{W}\right|$	&$(r,a)$	&BHCR Dual	\\ \midrule
14a	&(21,14,6,1;42)	&$x^2+y^3+z^7+w^{42}$	&1	&1	&(10,0)	&14a	\\
14b	&(21,14,6,1;42)	&$x^2+y^3+z^7+xw^{21}$	&1	&1	&(10,0)	&28b	\\
14d	&(21,14,6,1;42)	&$x^2+y^3+z^7+yw^{28}$	&1	&1	&(10,0)	&45b	\\
28a	&(10,7,3,1;21)	&$x^2w+y^3+z^7+xw^{11}$	&1	&1	&(10,0)	&28a	\\
28b	&(10,7,3,1;21)	&$x^2w+y^3+z^7+w^{21}$	&1	&1	&(10,0)	&14b	\\
28d	&(10,7,3,1;21)	&$x^2w+y^3+z^7+yw^{14}$	&1	&1	&(10,0)	&45c	\\
32	&(7,3,2,2;14)	&$x^2+y^4z+z^7+w^7$	&1	&1	&(4,1)	&35a	\\
35a	&(14,7,4,3;28)	&$x^2+y^4+z^7+yw^7$	&1	&1	&(16,1)	&32	\\
35b	&(14,7,4,3;28)	&$x^2+xy^2+z^7+yw^7$	&1	&1	&(16,1)	&66a\\
45a	&(14,9,4,1;28)	&$x^2+y^3w+z^7+yw^{19}$	&1	&1	&(10,0)	&45a\\
45b	&(14,9,4,1;28)	&$x^2+y^3w+z^7+w^{28}$	&1	&1	&(10,0)	&14d\\
45c	&(14,9,4,1;28)	&$x^2+y^3w+z^7+xw^{14}$	&1	&1	&(10,0)	&28d\\
66a	&(3,2,1,1;7)	&$x^2z+xy^2+z^7+w^7$	&1	&1	&(4,1)	&35b\\
66b	&(3,2,1,1;7)	&$x^2w+xy^2+z^7+yw^5$	&3	&3	&(16,1)	&66b\\
	&	&	&	&1	&(4,1)	&\\ \bottomrule
\end{tabular}
\caption[Table for $p=7$]{Calculations for $p=7$}\label{tab_eq7}

\end{table}


\begin{table}
\centering
\begin{tabular}{@{}lllcccc@{}}\toprule

\multicolumn{7}{c}{Table for $p=13$} \\

No.    &Weights    &Polynomial	&$\left|\SLgp{W}/J_{W}\right|$	&$\left|G/J_{W}\right|$	&$(r,a)$	&BHCR Dual	\\ \midrule
87	&(5,4,3,1;13)	& $x^2z+xy^2+yz^3+w^{13}$	&1	&1	&(10,1)	&87\\
\bottomrule
\end{tabular}
\caption[Table for $p=13$]{Calculations for $p=13$}\label{tab_eq13}

\end{table}

\end{alphafootnotes}

\clearpage

\appendix
\appendix
\section{Tables of lattices}\label{tables_lattices}

In Tables \ref{tab_lattices3} -- \ref{tab_lattices13}, we reproduce tables in \citep{order_three,Taki_order3} and \citep{other_primes} that classify possibilities for the following situation.  Given a K3 surface with a non-symplectic automorphism $\sigma$ of prime order $p\in\set{3,5,7,13}$, each table describes all possible values for the invariants $(r,a)$ of the fixed lattice $S(\sigma)$, the invariants $(g,n,k)$ of the fixed locus $X^\sigma$, the isometry class of the orthogonal complement $T(\sigma)$ of $S(\sigma)$, and the isometry class of $S(\sigma)$.

We also include diagrams, discussed at the end of Section \ref{sec_LPK3}, of the possible values of the invariants $(m,a)$ of $S(\sigma)$ for the primes $p=3,5,7$.  (Recall that $m=(22-r)/(p-1)$ is the $\Z[\zeta_p]$-rank of $T(\sigma_p)$.) Table \ref{tab_lattices3} is a modification, using \citep[Corollary 1.13.3]{nikulin_ISBF}, of the tables found in \citep{order_three} and \citep{Taki_order3} that makes the verification of mirror symmetry more straightforward. For instance, we have used in places the isometry $U(3)\oplus E_6\cong U\oplus A_2^3$.

\begin{table}[h!]
    \centering
\begin{tabular}{l|l|l|l|l|l|l}
$r$&$a$&$g$&$n$&$k$& $T(\sigma)$& $S(\sigma)$\\
\hline
2 &2 &4 &0 &0 & $U\oplus U\oplus A_2\oplus E_6\oplus E_8$  &$U(3)$\\
2 &0 &5 &0 &1 & $U\oplus U\oplus E_8^2$  & $U$\\
4 &3 &3 &1 &0 & $U\oplus U\oplus A_2\oplus E_6^2$ & $U(3)\oplus A_2$\\
4 &1 &4 &1 &1 & $U\oplus U\oplus E_6\oplus E_8$ & $U\oplus A_2$\\
6 &4 &2 &2 &0 & $U\oplus U\oplus A_2^3\oplus E_6$ & $U(3)\oplus A_2^2$\\
6 &2 &3 &2 &1 & $U\oplus U\oplus E_6^2$ & $U\oplus A_2^2$\\
8 &7 &- &3 &- & $U\oplus U(3)\oplus A_2^5$ & $U(3)\oplus E_6^*(3)$\\
8 &5 &1 &3 &0 & $U\oplus U\oplus A_2^5$ & $U(3)\oplus A_2^3$\\
8 &3 &2 &3 &1 & $U\oplus U\oplus A_2^2\oplus E_6$ & $U\oplus A_2^3$\\
8 &1 &3 &3 &2 & $U\oplus U\oplus A_2\oplus E_8$ & $U\oplus E_6$\\
10 &6 &0 &4 &0 & $U\oplus U(3)\oplus A_2^4$ & $U(3)\oplus A_2^4$\\
10 &4 &1 &4 &1 & $U\oplus U\oplus A_2^4$ & $U\oplus A_2^4$\\
10 &2 &2 &4 &2 & $U\oplus U\oplus A_2\oplus E_6$ & $U\oplus A_2\oplus E_6$\\
10 &0 &3 &4 &3 & $U\oplus U\oplus E_8$ & $U\oplus E_8$\\
12 &5 &0 &5 &1 & $U\oplus U(3)\oplus A_2^3$ & $U\oplus A_2^5$\\
12 &3 &1 &5 &2 & $U\oplus U\oplus A_2^3$ & $U\oplus A_2^2\oplus E_6$\\
12 &1 &2 &5 &3 & $U\oplus U\oplus E_6$ & $U\oplus A_2\oplus E_8$\\
14 &4 &0 &6 &2 & $U\oplus U(3)\oplus A_2^2$ & $U\oplus A_2^3\oplus E_6$\\
14 &2 &1 &6 &3 & $U\oplus U\oplus A_2^2$ & $U\oplus E_6^2$\\
16 &3 &0 &7 &3 & $U\oplus U(3)\oplus A_2$ & $U\oplus A_2\oplus E_6^2$\\
16 &1 &1 &7 &4 & $U\oplus U\oplus A_2$ & $U\oplus E_6\oplus E_8$\\
18 &2 &0 &8 &4 & $U\oplus U(3)$ & $U\oplus A_2\oplus E_6\oplus E_8$\\
18 &0 &1 &8 &5 & $U\oplus U$ & $U\oplus E_8^2$\\
20 &1 &0 &9 &5 & $A_2(-1)$ & $U\oplus E_8^2\oplus A_2$\\
\end{tabular}
\caption{Order $3$}
\label{tab_lattices3}
\end{table}

 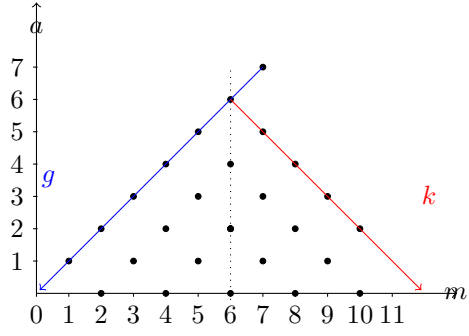
\begin{figure}[h]
\begin{tikzpicture}[scale=.43]
\filldraw [black] 
(1,1) circle (2.5pt) 
(2,0) circle (2.5pt)
(2,2) circle (2.5pt) 
(3,1) circle (2.5pt)
(3,3) circle (2.5pt)
(4,0) circle (2.5pt)
(4,2) circle (2.5pt)
(4,4) circle (2.5pt)
(5,1) circle (2.5pt)
(5,3) circle (2.5pt)
(5,5) circle (2.5pt)
(6,0) circle (2.5pt)
(6,2) circle (2.5pt)
(6,4) circle (2.5pt)
(6,2) circle (2.5pt)
(6,6) circle (2.5pt)
(7,1) circle (2.5pt)
(7,3) circle (2.5pt)
(7,5) circle (2.5pt)
(7,7) circle (2.5pt)
(8,0) circle (2.5pt)
(8,2) circle (2.5pt)
(8,4) circle (2.5pt)
(9,1) circle (2.5pt)
(9,3) circle (2.5pt)
(10,0) circle (2.5pt)
(10,2) circle (2.5pt)
 ; 
\draw[dotted] (6,-0.2)--coordinate (x axis mid) (6,7);
\draw[->] (0,0) -- coordinate (x axis mid) (13,0);
    \draw[->] (0,0) -- coordinate (y axis mid)(0,9);
    \foreach \x in {0,1,2,3,4,5,6,7,8,9,10,11}
        \draw [xshift=0cm](\x cm,0pt) -- (\x cm,-3pt)
         node[anchor=north] {$\x$};
          \foreach \y in {1,2,3,4,5,6,7}
        \draw (1pt,\y cm) -- (-3pt,\y cm) node[anchor=east] {$\y$};
    \node[below=0.2cm, right=2.5cm] at (x axis mid) {$m$};
    \node[left=0.2cm, below=-1.8cm] at (y axis mid) {$a$};
    \draw[<-, blue](0.1,0.1)-- node[below=1cm,left=1.15cm]{$g$} (7,7);   
 \draw[<-, red](11.9,0.1)-- node[below=1cm,right=1.15cm]{$k$} (6,6);
   \end{tikzpicture} 
\caption{Order $p=3$}
\end{figure}

\begin{table}[h!]
    \centering
\begin{tabular}{l|l|l|l|l|l|l}
$r$&$a$&$g$&$n$&$k$&$T(\sigma)$&$S(\sigma)$\\
\hline
2&1&2&1&0&$U\oplus H_5\oplus E_8^2$&$H_5$\\
6&2&1&4&0&$U\oplus H_5\oplus A_4\oplus E_8$&$H_5\oplus A_4$\\
6&4&-&4&-&$U(5)\oplus H_5\oplus A_4\oplus E_8$&$H_5\oplus A_4^*(5)$\\
10&1&1&7&1&$U\oplus H_5\oplus E_8$&$H_5\oplus E_8$\\
10&3&0&7&0&$U\oplus H_5 \oplus A_4^2$&$H_5\oplus A_4^2$\\
14&2&0&10&1&$U\oplus H_5\oplus A_4$&$H_5\oplus A_4\oplus E_8$\\
18&1&0&13&2&$U\oplus H_5$&$H_5\oplus E_8^2$\\
\end{tabular}
\caption{Order $5$}
\label{tab_lattices5}
\end{table}

\begin{table}[h!]
    \centering
\begin{tabular}{l|l|l|l|l|l|l}
$r$&$a$&$g$&$n$&$k$&$T(\sigma)$&$S(\sigma)$\\
\hline
4&1&1&3&0&$U\oplus U\oplus A_6\oplus E_8$&$U\oplus K_7$\\
4&3&-&3&-&$U\oplus U(7)\oplus A_6\oplus E_8$&$U(7)\oplus K_7$\\
10&0&1&8&1&$U\oplus U\oplus E_8$&$U\oplus E_8$\\
10&2&0&8&0&$U\oplus U(7)\oplus E_8$&$U(7)\oplus E_8$\\
16&1&0&13&1&$U\oplus U\oplus K_7$&$U\oplus A_6\oplus E_8$\\
\end{tabular}
\caption{Order $7$}
\label{tab_lattices7}
\end{table}

\begin{figure}[h!]
\begin{floatrow}
\ffigbox[\FBwidth]
{
\begin{tikzpicture}[scale=.51]
\filldraw [black] 
(1,1) circle (2.5pt)
(2,2) circle (2.5pt)
(3,1) circle (2.5pt)
(3,3) circle (2.5pt)
(4,2) circle (2.5pt)
(4,4) circle (2.5pt)
(5,1) circle (2.5pt)
 ; 
\draw [dotted] (3,-0.2) -- coordinate (x axis mid) (3,4);
\draw[->] (0,0) -- coordinate (x axis mid) (7,0);
    \draw[->] (0,0) -- coordinate (y axis mid)(0,5.5);
    \foreach \x in {0,1,2,3,4,5}
        \draw [xshift=0cm](\x cm,0pt) -- (\x cm,-3pt)
         node[anchor=north] {$\x$};
          \foreach \y in {1,2,3,4}
        \draw (1pt,\y cm) -- (-3pt,\y cm) node[anchor=east] {$\y$};
    \node[below=0.2cm, right=1cm] at (x axis mid) {$m$};
    \node[left=0.3cm, below=-1.2cm] at (y axis mid) {$a$};
 \draw[<-, blue](0.1,0.1)-- node[below=0.5cm,left=0.5cm]{$g$} (4,4);   
 \draw[<-, red](5.9,0.1)-- node[below=0.3cm,right=0.5cm]{$k$} (3,3);
  \end{tikzpicture}
}
{\captionsetup{margin={-10pt,-10pt},font=small}\caption{Order $p=5$}\label{fig:1}}%
\hspace{1.5cm}%
\ffigbox[\FBwidth]
{\begin{tikzpicture}[scale=.51]
\filldraw [black] 

(1,1) circle (2.5pt)  
(2,0) circle (2.5pt) 
(2,2) circle (2.5pt)
(3,1) circle (2.5pt)
(3,3) circle (2.5pt)
 ; 
\draw[dotted] (2,-0.2)--coordinate (x axis mid) (2,3);
\draw[->] (0,0) -- coordinate (x axis mid) (5,0);
    \draw[->] (0,0) -- coordinate (y axis mid)(0,4.5);
    \foreach \x in {0,1,2,3}
        \draw [xshift=0cm](\x cm,0pt) -- (\x cm,-3pt)
         node[anchor=north] {$\x$};
          \foreach \y in {1,2,3}
        \draw (1pt,\y cm) -- (-3pt,\y cm) node[anchor=east] {$\y$};
    \node[below=0.2cm, right=0.7cm] at (x axis mid) {$m$};
    \node[left=0.3cm, below=-1.0cm] at (y axis mid) {$a$};
 \draw[<-, blue](0.1,0.1)-- node[below=0.45cm,left=0.4cm]{$g$} (3,3);   
 \draw[<-, red](3.9,0.1)-- node[below=0.3cm,right=0.4cm]{$k$} (2,2);
  \end{tikzpicture} }
{\captionsetup{margin={-10pt,-10pt}, font=small}\caption{Order $p=7$}\label{fig:2}}
\end{floatrow}
\end{figure}

\begin{table}[h!]
	\centering
\begin{tabular}{l|l|l|l|l|l|l}
$r$&$a$&$g$&$n$&$k$&$T(\sigma)$&$S(\sigma)$\\
\hline
10&1&0&9&0&$U\oplus E_8\oplus H_{13}$&$E_8\oplus H_{13}$\\
\end{tabular}
\caption{Order $13$}
\label{tab_lattices13}
\end{table}

\clearpage
\bibliographystyle{abbrvnat}
\bibliography{references}

\end{document}